\documentclass[a4paper,12pt,intlimits,ngerman]{amsart}
\usepackage[unicode]{hyperref}
\usepackage[mathscr]{euscript}
\usepackage{esint}
\usepackage[margin=2cm]{geometry}
\usepackage{graphicx}
\usepackage{color}
\usepackage{comment}
\usepackage[notref,notcite,final]{showkeys}
\usepackage{cite}

\usepackage{relsize}


\newtheorem{theorem}{Theorem}[section]

\newtheorem{proposition}[theorem]{Proposition}
\newtheorem{lemma}[theorem]{Lemma}
\newtheorem{corollary}[theorem]{Corollary}

\newtheorem{definition}[theorem]{Definition}

\newtheorem{externaltheorem}{Theorem}

\hypersetup{
breaklinks=true,
colorlinks=true,
linkcolor=blue,
citecolor=blue,
urlcolor=blue,
}

\numberwithin{equation}{section}

\DeclareMathOperator{\supp}{supp}

\DeclareMathOperator{\dist}{dist}
\DeclareMathOperator{\conv}{conv}
\DeclareMathOperator{\diam}{diam}

\DeclareMathOperator{\bilip}{biLip}

\DeclareMathOperator{\essinf}{essinf}

\def\emob{{E_\textnormal{M\"ob}}}
\def\Rot{{\textnormal{Rot}}}

\newcommand\ANG{\mathop{\mbox{$<\!\!\!)$}}\nolimits}

\newcommand{\measurerestr}{%
 \,\raisebox{-.127ex}{\reflectbox{\rotatebox[origin=br]{-90}{$\lnot$}}}\,%
}
\newcommand{\omitted}[1]{}

\definecolor{colorph}{rgb}{0,.4,.6}

\usepackage{soul}

\setul{1ex}{.2ex}

\newcommand{\R}{\mathbb{R}}
\newcommand{\N}{\mathbb{N}}
\renewcommand{\S}{\mathbb{S}}
\newcommand{\Z}{\mathbb{Z}}

\newcommand\Id{{{\rm Id}}}

\newcommand{\HH}{\mathscr{H}}
\newcommand{\HL}{\mathscr{L}}

\definecolor{mygreen}{RGB}{30,50,0}

\newcommand{\Fo}{\,\,\,\text{for }\,\,}
\newcommand{\Foa}{\,\,\,\text{for all }\,\,}

\newcommand{\AND}{\,\,\,\text{and }\,\,}

\newcommand{\ON}{\,\,\,\text{on }\,\,}

\newcommand{\As}{\,\,\,\text{as }\,\,}

\newcommand{\g}{\gamma}
\newcommand{\G}{\Gamma}

\newcommand{\doi}[2]{}

\def\mvint_#1{\mathchoice
          {\mathop{\vrule width 6pt height 3 pt depth -2.5pt
            \kern -8pt \intop}\nolimits_{\kern -3pt #1}}%
             {\mathop{\vrule width 5pt height 3 pt depth -2.6pt
                        \kern -6pt \intop}\nolimits_{#1}}%
              {\mathop{\vrule width 5pt height 3 pt depth -2.6pt
                   \kern -6pt \intop}\nolimits_{#1}}%
              {\mathop{\vrule width 5pt height 3 pt depth -2.6pt
                      \kern -6pt \intop}\nolimits_{#1}}}

\title[Stability of knot equivalence and symmetric critical knots]
{Stability of knot equivalence at low regularity, 
and symmetric critical knots for the M\"obius energy}
\author{
Simon Blatt}
\address[Simon Blatt]{Paris London Universit\"at Salzburg, Department of Mathematics, Hellbrunner Strasse~34, 5020~Salzburg, Austria}
\email[Simon Blatt]{simon.blatt@sbg.ac.at}

\author{Alexandra Gilsbach}
\address[Alexandra Gilsbach]{
RWTH Aachen University,
Institute for Mathematics,
Templergraben~55,
52062~Aachen, Germany}
\email{gilsbach@instmath.rwth-aachen.de}

\author{Philipp Reiter}
\address[Philipp Reiter]{Chemnitz University of Technology,
Faculty of Mathematics, 09107~Chemnitz, Germany}
\email{reiter@math.tu-chemnitz.de}

\author{Heiko von der Mosel}
\address[Heiko von der Mosel]{
RWTH Aachen University,
Institute for Mathematics,
Templergraben~55,
52062~Aachen, Germany}
\email{heiko@instmath.rwth-aachen.de}

\date{\today}

\begin{document}

\begin{abstract}
We present 
sufficient criteria for the equivalence of tame knots at low regularity.
To this end, we introduce a localized version of Gromov's distortion
for any closed path-connected subset of $\R^n$. If two such sets have
local Gromov distortion
below a universal dimension-dependent constant $g_n$ at some scale, and if their
Hausdorff-distance is less than one quarter of that scale, 
we can show
that the fundamental groups of their complements are isomorphic. 
In addition, we construct this
isomorphism so that it restricts to the corresponding
peripheral subgroups as an isomorphism as well.

Applied to the images of one-dimensional knots
it follows that two knots are equivalent if their Hausdorff-distance
is bounded in terms of the scale under which their local 
Gromov distortion is controlled. From that we deduce
novel stability results for knot
equivalence in the Lipschitz category, and in the setting of fractional
Sobolev regularity below $C^1$. Moreover,
we prove a compactness theorem of knot equivalence classes
with respect to weak $W^{3/2,2}$-convergence.

As an application we show that the M\"obius energy introduced by 
O'Hara~\cite{ohara_1991a} can be minimized within arbitrary
prime knot 
classes under a symmetry constraint, and 
that these minimizers are in fact critical points and therefore
smooth and even
real analytic.  In particular, in every torus knot class there are at
least two distinct critical knots for the M\"obius energy.
\end{abstract}

\maketitle

\tableofcontents


\section{Introduction}

\subsection{Motivation}
The notion of \emph{ambient isotopy} is used in classical knot theory
to distinguish between different knot classes. It follows from a result
of Fisher \cite[Theorem~15]{fisher_1960} that 
two knots $\g_1,\g_2\in C^0(\R/\Z,\R^3)$
are ambiently isotopic if there exists an orientation preserving
homeomorphism $h:\R^3\to\R^3$ such that $h\circ\g_1=\g_2$. In concrete applications
it might be hard to come up with such a homeomorphism, but in the smooth category one
has the following useful stability theorem which to the best of our knowledge
has first been explicitly stated and proven by Diao et al.\@ who give also credit
to Ken Millet.
\begin{externaltheorem}[$C^1$-isotopy stability {\cite[Lemma 3$\cdot$2]{diao-etal_1999},
\cite{reiter_2005}, \cite[Proposition 3.1]{denne-sullivan_2008}}]
\label{thm:C1-stability}
For any immersed embedded curve $\g\in C^1(\R/\Z,\R^3)$ there is some number
$\epsilon_\g>0$ such that all curves $\eta\in C^1(\R/\Z,\R^3)$ with 
$\|\g-\eta\|_{C^1} < \epsilon_\g$ are embedded and ambiently isotopic to $\g$.
\end{externaltheorem}
This result also follows from the fact that any $C^1$-isotopy can be extended
to an ambient isotopy in arbitrary dimensions and co-dimensions which was proven\footnote{
For a proof in the simpler setting of the $C^2$-category see 
\cite[Chapter 8]{hirsch_1994}.}
by the first author in \cite{blatt_2009a}. The $C^1$-stability of isotopy is used
quite frequently in \emph{geometric knot theory} where one produces and investigates
specific knot representatives as minimizers or critical points of so-called
knot energies, 
to obtain more information about their knot classes; see, e.g., the
surveys 
\cite{strzelecki-vdm_2013a,strzelecki-etal_2013a,strzelecki-vdm_2014,blatt-reiter_2014,strzelecki-vdm_2018}  
and the many references therein. 
However, for the most interesting and mathematically more challenging
scale-invariant knot energies such as the M\"obius energy\footnote{In addition to 
scale-invariance the M\"obius energy is also invariant under M\"obius transformations
of the ambient Euclidean space \cite[Theorem~2.1]{freedman-etal_1994}, which lent
this energy that name.}
introduced by O'Hara \cite{ohara_1991a}
the underlying energy spaces do \emph{not} embed into $C^1$. Therefore,
one cannot expect, e.g., that minimizing sequences for these energies subconverge
in $C^1$, so that it is a priori not clear if the limit curves still belong to the
prescribed knot class. One of our goals in the present paper is therefore
to develop criteria
for stability of knot classes at low regularity.

For that we consider the slightly weaker notion
of \emph{knot equivalence} instead of ambient isotopy. Two knots $\g_1,\g_2\in
C^0(\R/\Z,\R^3)$ are said to be \emph{equivalent}, denoted by $\g_1\sim\g_2$, if and
only if there is a homeomorphism $h:\R^3\to\R^3$ such that $h(\g_1(\R/\Z))=\g_2(\R/\Z)$;
cf.\@ \cite[p. 4]{crowell-fox_1977}. It is known, e.g., that the trefoil knot is not
ambiently isotopic to its mirror image; see \cite[Theorem~3.39 (b)]{burde-zieschang_2014}. The corresponding reflection, however,
is an orientation reversing homeomorphism of $\R^3$ mapping the image of the trefoil
knot onto its mirror image, so that the two knots are equivalent.  So, there are
more (ambient) isotopy types than knot equivalence classes in knot space. On the other
hand, if two tame
arclength parametrized knots are equivalent, then one
can show that they are either ambiently
isotopic, or one is ambiently isotopic to the other one's mirror image. 

Gordon and Luecke studied in \cite{gordon-luecke_1989}
the fundamental groups of the knots' complements, also 
called \emph{knot groups}, and their peripheral subgroups. Combining their work
with results of Waldhausen \cite{waldhausen_1968} yields the following sufficient
criteria for knot equivalence.
\begin{externaltheorem}[Knot equivalence via knot complements
{\cite{waldhausen_1968,gordon-luecke_1989}}]\label{thm:gordon-luecke}
Let  $\g_1,\g_2\in C^0(\R/\Z,\R^3)$ be two tame knots.
\begin{enumerate}
\item[\rm (i)] If their complements $\R^3\setminus\g_i(\R/\Z)$, $i=1,2$, are homeomorphic,
then $\g_1\sim\g_2$.
\item[\rm (ii)] If $\g_1,\g_2$ are prime and their knot groups 
$\pi_1(\R^3\setminus\g_i(\R/\Z))$, $i=1,2,$ are isomorphic,  then $\g_1\sim\g_2$. 
\item[\rm (iii)] If the knot groups $\pi_1(\R^3\setminus\g_i(\R/\Z))$, $i=1,2,$ 
are isomorphic and if this isomorphism induces isomorphic peripheral subgroups $P_{\g_i}$, $i=1,2,$
then $\g_1\sim\g_2$ even if the $\g_i$ are composite knots.
\end{enumerate}
\end{externaltheorem}

Recall that for a set $M\subset\R^n$ with path-connected complement $\R^n\setminus M$
two closed loops $c_1,c_2\in C^0(\R/\Z,\R^n\setminus M)$ are equivalent if they are
homotopic, and that the fundamental group $\pi_1(\R^n\setminus M)$ consists of
the  respective equivalence classes. The group  structure is induced by concatenation
of closed loops containing a fixed base point in $\R^n\setminus M$. The peripheral
subgroup $P_M\subset\pi_1(\R^n\setminus M)$ then consists of all equivalence classes that
possess representative loops arbitrarily close to $M$.

\subsection{Isomorphic fundamental groups via localized distortion}
In order to apply the results on knot equivalence, i.e., Theorem~\ref{thm:gordon-luecke},
under low regularity assumptions
we are going to use a localized version of Gromov's distortion \cite{gromov_1978}. 
For a closed, path-connected set $M\subset\R^n$ its \emph{distortion} is defined as
the quantity
\begin{equation}\label{eq:distortion}
\delta(M):=\sup_{x,y\in M\atop x\not= y}\frac{d_M(x,y)}{|x-y|}
\qquad\ge 1,
\end{equation}
where $d_M(x,y)$ denotes the  
\emph{intrinsic distance} between the points $x,y$, i.e.,
the length of the shortest curve  in $M$ connecting $x$ with $y$. It is well-known that
the circle $c_{2\pi}$ satisfies $\delta(c_{2\pi})=\frac{\pi}2$ uniquely minimizing 
the distortion among all closed curves (see \cite[Proposition 2.1]{kusner-sullivan_1997} 
for a nice and short proof of that fact).  
Moreover, any non-trivial tame
knot $\g\in C^0(\R/\Z,\R^3)$ has distortion at least
$\frac{5\pi}3$, which was shown by Denne and Sullivan in \cite{denne-sullivan_2009}.
So, distortion detects knottedness, but it is not useful to distinguish between different
knots. For example, distortion cannot prevent the pull-tight phenomenon 
where knotted subarcs of a curve 
can be scaled down to vanish in the limit without any effect
on the curve's distortion \cite{mullikin_2007}. In addition, O'Hara constructed infinitely many
different prime knots with a uniform upper bound on their distortion 
\cite[Theorem~3.5]{ohara_1992a}.
On the other hand, Gromov asked if there is a universal bound on the distortion 
under which \emph{every}
knot type can be represented.
This was answered to the negative by Pardon in
\cite[Theorems 1.1 \& 1.2]{pardon_2011}
who also points out~\cite[p.~638]{pardon_2011} that there are 
even wild knots with finite distortion such as \cite[Figure~5]{crowell-fox_1977}.

Nevertheless, a universal bound on a
localized version of distortion
together with a comparable $L^\infty$-distance guarantees knot equivalence as
we will see in Corollary \ref{cor:equiv-knots} below.
\begin{definition}\label{def:local-distortion}
For a nonempty closed, path-connected set $M\subset\R^n$ and $r>0$ we define
the \emph{local distortion of $M$ at scale $r$} as
\begin{equation}\label{eq:local-distortion}
\delta(M,r):=\sup_{x,y\in M\atop 0<|x-y|\le 2r}\frac{d_M(x,y)}{|x-y|}.
\end{equation}
We say that \emph{$M$ has local distortion strictly below $\varepsilon>0$}
if there is some $r_M>0$ such that $\delta(M,r_M)<\varepsilon$,
and we call any such $r_M>0$ an \emph{admissible scale}.
\end{definition}

Our first result on fundamental groups is stated for general 
subsets of $\R^n$ and might therefore
be of independent interest. Its formulation uses the
dimension-dependent constant
\begin{equation}\label{eq:dekster-jung}
g_n:=\sqrt{\frac{2n-2}n}\arcsin\sqrt{\frac{n}{2n-2}}
\quad\Fo n\ge 3,
\end{equation}
converging in a strictly decreasing manner
to $g_\infty:=\frac{\pi}{\sqrt{8}}$ as $n\to\infty$. Notice 
that
$g_3= \frac2{\sqrt{3}}\arcsin\frac{\sqrt{3}}2
= \frac{\sqrt{32}}{\sqrt{27}}g_\infty$ equals the
distortion of one third of a circle, whereas $g_\infty$
is the distortion of a quarter circle.
\begin{theorem}[Isomorphic fundamental groups]\label{thm:isomorphic}
Let $n\ge 3$.
If two nonempty closed, path-connected sets $M_i\subset\R^n$ 
with path-connected complements
$\R^n\setminus M_i$ have local distortion strictly
below $g_n$
at scales $r_{M_i}>0$ for $i=1,2$, and satisfy
\begin{equation}\label{eq:hausdorff}
\dist_\HH(M_1,M_2)<\frac14 \min\{r_{M_1},r_{M_2}\},
\end{equation}
then there exists an isomorphism $J:\pi_1(\R^n\setminus M_1)\to
\pi_1(\R^n\setminus M_2)$ such that its restriction to the peripheral
subgroup $P_{M_1}\subset \pi_1(\R^n\setminus M_1)$ maps $P_{M_1}$ isomorphically onto 
$P_{M_2}\subset \pi_1(\R^n\setminus M_2)$.
\end{theorem}

\subsection{Stability of knot equivalence}
As an immediate consequence of Theorem~\ref{thm:isomorphic} in combination with
part (iii) of Theorem~\ref{thm:gordon-luecke} we obtain a simple
sufficient geometric criterion for knot equivalence.
\begin{corollary}[Knot equivalence via local distortion]
\label{cor:equiv-knots}
Two tame knots $\g_i\in C^0(\R/\Z,\R^3)$ with local 
distortion 
 $\delta(\g_i(\R/\Z),r_{\g_i})<g_3$
at scales $r_{\g_i}>0$
for $i=1,2$, and
$\dist_{\HH}(\g_1(\R/\Z),\g_2(\R/\Z))$
$<\frac 14 \min\{r_{\g_1},r_{\g_2}\}$, are equivalent.
\end{corollary}

Notice that a small Hausdorff-distance, even an arbitrarily
small $L^\infty$-distance $\|\g_1-\g_2\|_{L^\infty}$
\emph{alone} does not guarantee knot equivalence:
One can insert small knotted arcs at any scale to produce highly complex
knots arbitrarily close to the round circle. Essential 
is the combination of controlled local
distortion at some scales and  
the Hausdorff-distance bounded in terms of these scales.
With Corollary \ref{cor:equiv-knots} we can prove 
the following stability of knot equivalence in the 
Lipschitz category.

\begin{corollary}[Lipschitz stability of knot equivalence]\label{cor:lipschitz-stability}
Let $\g:\R/\Z\to\R^3$ be an arclength parametrized tame 
knot with 
local distortion $\delta(\g(\R/\Z),r_\g)< g_3$
at scale $r_\g>0$. Then for all $h\in C^{0,1}(\R/\Z,\R^3)$
there exists $\epsilon_{\g,h}>0$ such that $\g_\epsilon:=\g+\epsilon h$ is
an immersed knot equivalent to $\g$ for all $|\epsilon| \le \epsilon_{\g,h}$.
\end{corollary}
Note that Lipschitz continuous curves are differentiable a.e., and we call such curves
\emph{immersed} if $v_\g:=\essinf_{u\in\R/\Z}|\g'(u)|>0$.

For $C^1$-knots $\g$ one has, of course, $\delta(\g(\R/\Z),r)\to 1 $ as $r\to 0$, so that
one finds a scale $r_\g$ sufficiently small so that $\g$ has 
local distortion strictly below $g_3$.
In fact, this observation implies a weaker version of Theorem~\ref{thm:C1-stability} where ambient isotopy is replaced by knot equivalence.
But we will prove here
that a fractional Sobolev regularity (below $C^1$) suffices, 
which leads to the following fractional stability result for knot equivalence.
\begin{corollary}[Fractional stability of knot equivalence]
\label{cor:fractional-stability}
For any arclength parametrized knot $\g\in W^{\frac32,2}(\R/\Z,\R^3)$ there exists
$\epsilon_\g>0$ such that every arclength parametrized
curve $\eta\in B_{\epsilon_\g}(\g)\subset
W^{\frac32,2}(\R/\Z,\R^3)$ is a knot equivalent to $\g$.
\end{corollary}
The preceding statement can be shown to hold true also
if $\g$ and
$\eta$ are only immersed and not necessarily arclength
parametrized, but we will not discuss the details here.
Open at this point, however, is if the bound on local 
distortion in Corollaries \ref{cor:equiv-knots} and 
\ref{cor:lipschitz-stability} already
implies tameness of the knots, or some improved regularity.

\subsection{Symmetric critical points for scale-invariant knot energies}
The fractional Sobolev space $W^{\frac32, 2}(\R/\Z,\R^3)$ is the natural
energy space for  the 
\emph{M\"obius energy}
\begin{equation}\label{eq:emob}
E_\textnormal{M\"ob}(\g):=\int_{\R/\Z}\int_{\R/\Z}\left(
\frac1{|\g(s)-\g(t)|^2}-
\frac{1}{d_\g(\g(s),\g(t))^2}\right)|\g'(s)||\g'(t)|\,dsdt.
\end{equation}
Indeed, an arclength parametrized
curve $\g:\R/\Z\to\R^3$ has finite M\"obius energy  
if and only if 
$\g$ is embedded and of
class $W^{\frac32,2}(\R/\Z,\R^3)$ as shown in \cite{blatt_2012a} 
(see 
\cite[Theorem~3.2 (i)]{gilsbach-vdm_2018} for an explicit a priori bound
on the fractional seminorm in terms of  the energy). 
Freedman et al.\@ managed to
use the M\"obius invariance of $\emob$ to find minimizing knots in any given prime
knot class \cite[Theorem~4.3]{freedman-etal_1994}.
Critical points of the M\"obius energy were
shown to be smooth and even real-analytic 
\cite{blatt-etal_2016,blatt-vorderobermeier_2019} 
without using the invariance of $\emob$ under 
M\"obius transformations. 
But how can we find critical knots  possibly different from the absolute minimizers? 
One option is to use symmetry as carried out in \cite{gilsbach-vdm_2018} for 
O'Hara's
energies above scale-invariance whose energy spaces \emph{do} 
embed into $C^1$.

As in that 
paper we consider here for the M\"obius energy a discrete rotational symmetry. To be
more precise, for $p\in\N$, $p\ge 2$, we call a 
curve $\g\in C^0(\R/\Z,\R^3)$
\emph{($p$-rotationally) symmetric} if and only if
\begin{equation}
\label{eq:rot-symm}
\textstyle
\Rot_{\frac{2\pi}p}\circ\g(s -\frac1{p})=\g(s)\quad\Foa s\in\R/\Z, 
\end{equation}
where $\Rot_{\beta}\in SO(3)$ denotes the rotation about a fixed axis, say,
the $e_3$-coordinate axis, with rotation angle $\beta$.
Define for a given knot equivalence class $\mathcal{K}$ the 
\emph{($p$-rotationally) symmetric
subset}
\begin{equation}\label{eq:symmetric-subset}
\Sigma_p(\mathcal{K}):=\{\eta\in C^{0,1}(\R/\Z,\R^3):|\eta'|>0
\,\,\textnormal{a.e.},\, [\eta]=\mathcal{K},\,
\eta\,\,\textnormal{satisfies}\,\eqref{eq:rot-symm}\}.
\end{equation}
\begin{theorem}[Symmetric critical prime knots]
\label{thm:symm-crit-prime}
If the symmetric subset $\Sigma_p(\mathcal{K})$
of a prime knot equivalence class $\mathcal{K}$ contains a knot with
finite M\"obius energy or, equivalently, a 
$W^{3/2,2}$-knot,
then
there exists an arclength parametrized knot $\g\in C^\omega(\R/\Z,\R^3)
\cap\Sigma_p(\mathcal{K})$
such that
\begin{equation}\label{eq:minimizing}
\emob(\g)=\inf_{\Sigma_p(\mathcal{K})}\emob(\cdot),
\end{equation}
and  $\g$ is a critical knot, i.e., $d\emob (\g)[h]=0$ for all
$h\in W^{\frac32,2}(\R/\Z,\R^3)$. Moreover,
all torus knot classes $\mathcal{T}(a,b)$ for co-prime $a,b\in\Z\setminus\{0,\pm 1\}$
contain at least two different critical knots for the M\"obius energy.
\end{theorem}

One may prove a similar result for other scale-invariant knot energies
with energy spaces $W^{1+1/p,p}$
such as those from O'Hara's family; cf.\@ \cite{ohara_1992a,
blatt-etal_2022b}.

It might be easier to check the symmetries for polygonal
curves, so instead of a knot with finite M\"obius energy
we might require in Theorem \ref{thm:symm-crit-prime} that
$\Sigma_p(\mathcal{K})$ contains a polygonal knot.
Indeed, 
smoothing the corners such that the symmetry is preserved 
yields an immersed embedded smooth
curve in $\Sigma_p(\mathcal{K})$
which has finite Möbius energy due 
to~\cite{blatt_2012a} and the energy's invariance under reparametrization. We do not know, on the other hand, if it suffices
to assume in Theorem \ref{thm:symm-crit-prime} only
that $\mathcal{K}$ is tame and that
$\Sigma_p(\mathcal{K})$ is nonempty.

We mentioned before that two equivalent arclength parametrized tame
knots $\g_1,\g_2:\R/\Z\to\R^3$ are either already ambiently isotopic, or
$\g_1$ is ambiently isotopic to the mirror image $\g_2^*$ of $\g_2$.
The $O(3)$-invariance of the M\"obius energy therefore
allows us to 
strengthen Theorem~\ref{thm:symm-crit-prime} to ambient isotopy
classes.
\begin{corollary}[Symmetric critical prime knots in ambient isotopy
classes] \label{cor:symm-crit-prime}
Let $\mathcal{K}$ be an ambient isotopy class of prime knots and $\Sigma_p(\mathcal{K})$ be the class of $p$-rotationally symmetric regular Lipschitz
curves whose ambient isotopy class equals $\mathcal{K}$. If
$\Sigma_p(\mathcal{K})$ contains a knot of
finite M\"obius energy, then there is an arclength parametrized
minimizing real analytic knot in $\Sigma_p(\mathcal{K})$ such that
\eqref{eq:minimizing} holds, and this knot is $\emob$-critical. If,
more specifically,
$\mathcal{K}$ equals the ambient isotopy class of torus knots,
$\mathcal{T}(a,b)$ for co-prime $a,b\in\Z\setminus\{0,\pm 1\}$, then there
are at least two such $\emob$-critical knots.
\end{corollary}
Kim and Kusner use in \cite{kim-kusner_1993}
symmetric criticality to construct $\emob$-critical torus 
knots (even with explicit formulas for their energy values). 
Their numerical experiments suggest, that
for (co-prime) parameters $a,b$ with $|a|,|b|>2$ 
these critical
points seize to be stable. Analytically, this remains an
open question.

\subsection{Strategy and outline of the paper}\label{sec:strategy}
In Section~\ref{sec:distortion} we analyze the effects
of the smallness
condition on local distortion  on general
closed, path-connected subsets $M\subset\R^n$. For 
each $x$ near (but not in) $M$ we find
a direction that ``points away from $M$''. These
directions are glued together in Section~\ref{sec:pseudo}
with a partition of unity to construct a kind of
pseudo-gradient vector field for the distance function $\dist(\cdot,M)$.  We solve
uniquely and globally two different flow equations with respect
to this vector field to define homotopies in the complement
$\R^n\setminus M$ of $M$. These mappings are then used in Section~\ref{sec:iso} 
for $M:=M_1$ and $M_2$ in combination with the sets' 
controlled local distortion
to construct  an
isomorphism between the fundamental groups
$\pi_1(\R^n\setminus M_1)$ and $\pi_1(\R^n\setminus M_2)$ as stated in 
Theorem~\ref{thm:isomorphic}. 

We use in Section~\ref{sec:stability} the immediate Corollary
\ref{cor:equiv-knots} to prove the Lipschitz stability and fractional stability
of knot equivalence, i.e., Corollaries \ref{cor:lipschitz-stability} and
\ref{cor:fractional-stability}. The first one relies on the fact that for curves the smallness condition on local 
distortion  
is stable with respect to Lipschitz perturbations. 
For
the fractional stability we recall in Corollary \ref{cor:seminorm-bilip}
that a small local fractional semi-norm of the
tangent implies a local bilipschitz estimate, which can be turned into a local
distortion bound for embeddings.

More analytical work is required for the existence
of symmetric $\emob$-critical prime knots as stated in
Theorem~\ref{thm:symm-crit-prime}. Since minimizing sequences converge only
weakly in $W^{\frac32,2}$ to a 
limit curve, we can \emph{not} relie on the (strong) 
$W^{\frac32,2}$-stability of Corollary \ref{cor:fractional-stability} to
guarantee the prescribed knot equivalence class in the limit. Nevertheless, 
the limiting
curve $\g$ has finite M\"obius energy because of the energy's lower-semicontinuity, 
so that we have
an arclength parametrized tame
$W^{\frac32,2}$-knot $\g$ as a limit. Exactly as in the proof of
Corollary \ref{cor:fractional-stability} we therefore find a scale  at which
$\g$ has its local distortion strictly
below the universal constant
$g_3$. The same holds true for all members of
the minimizing sequence, because each of them is an arclength parametrized
$W^{\frac32,2}$-knot as well. But the $L^\infty$-convergence of that sequence
towards $\g$ might be too slow to satisfy the assumptions of Corollary \ref{cor:equiv-knots}
for any sequence member.
In other words, the scales at which 
the members of the minimizing sequence have local 
distortion below $g_3$
might decay to zero much faster than their $L^\infty$-distance to $\g$. 
The analytic reason for that behaviour lies in the fact that we do \emph{not}
have uniform absolute continuity of the integrals defining the fractional
semi-norms of the tangents of the minimizing sequence. So, we are led to 
study possible concentrations in these semi-norms to cut them out, and replace
the corresponding subarcs of the curves by straight segments. Since such a substitution
generally destroys the $W^{\frac32,2}$-regularity\footnote{\label{foot:harmonic-substitution}For arclength parametrized embedded curves, $W^{\frac32,2}$-regularity is equivalent to finite Möbius energy according to~\cite{blatt_2012a}.
Due to~\cite[Corollary~1.3]{freedman-etal_1994}, such a curve has a local bilipschitz constant (cf.\@ Footnote~\ref{foot:bilip}), arbitrarily close to~$1$.
However, replacing a subarc of a curve 
by a line is likely to produce sharp corners at its endpoints
which results in a bilipschitz constant strictly larger 
than~$1$.},
a careful analytic
investigation of that
procedure based on some subtle BMO-type estimates
is required in order to keep the local
distortion of these modified curves under control. This analysis is carried out
in Section~\ref{sec:harmonic}.

In Section~\ref{sec:compactness} we interpret
the fractional semi-norms of sequences of $W^{\frac32,2}$-knots as Radon measures
on $\R/\Z\times\R/\Z$ to pass to a weak limiting  Radon measure in order to analyze
its set of concentration points.  At  a carefully chosen scale around these 
concentration points we apply the cutting procedure of 
Section~\ref{sec:harmonic}
to members of the minimizing sequence with sufficiently large
index. The result
is a weak fractional compactness theorem in the set of symmetric, arclength 
parametrized $W^{\frac32,2}$-knots (see Theorem~
\ref{thm:fractional-compactness}), which
might be of independent interest. This is the main tool in the existence proof
for symmetric $\emob$-minimizing knots in Section~\ref{sec:symm-critical}. A simple
symmetrization argument which  uses the invariance of the M\"obius energy under rotations
then leads to the fact that these minimizing knots are also critical points for
$\emob$, and therefore are real analytic. This
finalizes the proof of Theorem~\ref{thm:symm-crit-prime}.

\section{Sets with local distortion below a quarter circle}\label{sec:distortion}
Let $M\subset\R^n$ for $n\ge 3$ 
be a nonempty closed path-connected set 
with local
distortion strictly below the dimension-dependent 
constant $g_n$ defined in
\eqref{eq:dekster-jung}.
Let $r_M=r_M(n)$ be an admissible scale.
Since the function $g(x):=\frac{x/2}{\sin(x/2)}
$ is strictly increasing on $[0,\pi]$ with 
$\lim_{x\searrow 0}g(x)=1$ and
$\lim_{x\nearrow \pi}g(x)=\frac{\pi}2>g_n>
\frac{\pi}{\sqrt{8}}>1$ for all $n\in\N,$
$n\ge 3$, there exist unique angles $0<\alpha_M<\beta_n<\pi $ such that
$1\le g(\alpha_M)=\delta(M,r_M)<g_n=g(\beta_n)<\frac{\pi}2$. 
We call $\alpha_M$
(which depends on the (non-unique) scale $r_M$)
the \emph{distortion angle}
of $M$, and by direct computation, 
\begin{equation}\label{eq:beta_n}
\beta_n=2\arcsin\sqrt{\frac{n}{2n-2}}.
\end{equation}

Our first lemma bounds the set of directions generated by a 
given point $x\in\R^n
\setminus M$ with $0<d_x:=\dist(x,M)<r_M$ and all points $z\in M$ close to $x$.
More precisely, abbreviate for $\epsilon >0$ this nonempty 
set of directions by
\begin{equation}\label{eq:set-directions}
\Sigma_\epsilon(x)
:=\Big\{\frac{x-z}{|x-z|}:z\in B_{(1+\epsilon)d_x}(x)\cap M\Big\}\subset\S^{n-1}.
\end{equation}
\begin{lemma}[Diameter bound on set of directions]
\label{lem:set-directions}
Let $M\subset\R^n$ be a non-empty, closed and path-connected set satisfying
$\delta(M,r_M)<g_n$ at some scale $r_M>0$ with distortion angle $\alpha_M
=g^{-1}(\delta(M,r_M))\in 
(0,\beta_n)$.
Then for each $\alpha\in (\alpha_M,\beta_n)$ 
there is some $\epsilon_0=\epsilon_0(\alpha)
\in (0,1)$ such that
\begin{equation}\label{eq:diam-sigma}
\diam_{\S^{n-1}}\Sigma_\epsilon(x) <\alpha\quad\Foa \epsilon\in (0,\epsilon_0),\,
x\in B_{\frac{r_M}{1+\epsilon}}(M)\setminus M.
\end{equation}
\end{lemma}
Before proving this notice that the spherical version of Jung's theorem 
\cite[Theorem~2]{dekster_1995}
implies 
that $\Sigma_\epsilon(x)$ is contained in a geodesic ball
of radius
\begin{equation}\label{eq:geodesic-radius}
R(\alpha):=\arcsin\Big(\sqrt{\frac{2n-2}n}\sin\frac{\alpha}2\Big)<\frac{\pi}2.
\end{equation}
\begin{corollary}[Common direction away from $M$]
\label{cor:common-direction}
Under the assumptions of Lemma \ref{lem:set-directions} one finds for each
$\alpha\in (\alpha_M,\beta_n)$ and
$\epsilon\in (0,\epsilon_0(\alpha))$ and for
every $x\in B_{\frac{r_M}{1+\epsilon}}(M)\setminus M$
some unit vector $e_x\in\S^{n-1}$ such that $\Sigma_\epsilon(x)$ is contained
in the geodesic ball on $\S^{n-1}$ of radius $R(\alpha)$  centered at $e_x$. 
In particular,
\begin{equation}\label{eq:common-direction}
\Big\langle \frac{x-z}{|x-z|},e_x\Big\rangle_{\R^n}\ge \cos R(\alpha)>0
\quad\Foa z\in B_{(1+\epsilon)d_x}(x)\cap M.
\end{equation}
\end{corollary}
{\it Proof of Lemma \ref{lem:set-directions}.}\,
Assume that the assertion is wrong.
Then there is some $\alpha\in(\alpha_M,\beta_n)$ such that
for any $\epsilon_0\in(0,1)$ we may find some
$\epsilon\in(0,\varepsilon_0)$
and some $x\in B_{\frac{r_M}{1+\epsilon}}(M)\setminus M$
such that $\diam_{\mathbb S^n}\Sigma_\epsilon(x)\ge\alpha$.
In particular, 
$d_x=\dist(x,M)>0$.
Thus, we may find $z_1,z_2\in B_{(1+\epsilon)d_x}(x)\cap M$ 
with
\begin{equation}\label{eq:large-angle}
\textstyle\pi\ge\chi:=\ANG\big(\frac{x-z_1}{|x-z_1|},
\frac{x-z_2}{|x-z_2|}\big)\ge\alpha.
\end{equation}
Since $B_{d_x}(x)\cap M=\emptyset$, the images of all curves
 $c:[0,1]\to M$ with $c(0)=z_1$ and $c(1)=z_2$ are  contained in 
 $\R^n\setminus B_{d_x}(x)$ so that their lengths $\HL(c)$ satisfy by 
 our assumption \eqref{eq:large-angle}
\begin{equation}\label{eq:lower-length-bound} 
 \HL(c)\ge \HH^1(c([0,1])\ge \HH^1(P_{d_x}\circ c([0,1])\ge\dist_{\partial B_{d_x}(x)}
 \big(P_{d_x}(c(0)),P_{d_x}(c(1))\big)
 \overset{\eqref{eq:large-angle}}{\ge}\chi d_x,
\end{equation} 
 where $P_{d_x}(z):=x+d_x(z-x)/|z-x|$ for $z\in\R^n\setminus B_{d_x}(x)$ denotes
 the next-point projection onto the sphere $\partial B_{d_x}(x)$. Notice that this
 projection is Lipschitz continuous with Lipschitz constant $1$
 on $\mathbb R^n\setminus B_{d_x}(x)$, thus decreasing
 the Hausdorff-measure, which 
 leads to the second inequality in \eqref{eq:lower-length-bound}
 due to \cite[Theorem~2.8]{evans-gariepy_2015}.
Now \eqref{eq:lower-length-bound}  implies $d_M(z_1,z_2)\ge\chi d_x$, whereas
 \begin{align*}
 |z_2-z_1|&\textstyle
 \le d_x\big|\frac{x-z_1}{|x-z_1|}-
 \frac{x-z_2}{|x-z_2|}\big|+|x-z_1|\big|1-\frac{d_x}{|x-z_1|}\big|+
 |x-z_2|\big|1-\frac{d_x}{|x-z_2|}\big|\\
 &\textstyle
 \le 2d_x\sin\frac{\chi}2+2\epsilon d_x\le 2(1+\epsilon)d_x<2r_M,
 \end{align*}
 which for $0<\epsilon<\epsilon_0$ for sufficiently small $\epsilon_0=
 \epsilon_0(\alpha)$ leads to the contradiction
 $$
\textstyle 
\delta(M,r_M)\ge
 \frac{d_M(z_1,z_2)}{|z_2-z_1|}\ge\frac{\chi/2}{\sin\frac{\chi}2 +\epsilon}
 \ge\frac{\alpha/2}{\sin\frac{\alpha}2 +\epsilon}>
 \frac{\alpha_M/2}{\sin\frac{\alpha_M}2}=g(\alpha_M)=\delta(M,r_M),
 $$
 where we also used that the function $s\mapsto\frac{s/2}{\sin\frac{s}2 +\epsilon}$
 is non-decreasing on $[0,\pi]$ for any $\epsilon>0$.
 \qed

 In Corollary \ref{cor:common-direction} we established for every $x\in
 B_{\frac{r_M}{1+\epsilon}}(M)\setminus M$ a common direction $e_x\in\S^{n-1}$
 pointing away from $M$. This unit vector keeps this property under small
 perturbations of the base point $x$.
 \begin{corollary}[Locally uniform common direction]
 \label{cor:loc-unif-common-direction}
Let $M\subset\R^n$ be a non-empty, closed and path-connected set with
$g(\alpha_M)=\delta(M,r_M)<g_n$ at some scale $r_M>0$ with distortion angle
$\alpha_M\in (0,\beta_n)$ and $\beta_n$ as in \eqref{eq:beta_n}.
 Then for each $\alpha\in (\alpha_M,\beta_n)$ there is $\epsilon_0=\epsilon_0(\alpha)>0$
such that the following holds: For all $\epsilon\in (0,\epsilon_0)$ and
$x\in B_{\frac{r_M}{1+\epsilon}}(M)\setminus M$ there exists $e_x\in\S^{n-1}$ and
$\sigma_0=\sigma_0(\alpha,\epsilon)\in (0,1)$ 
such that for all $\sigma\in (0,\sigma_0]$
one has the inclusion
\begin{equation}\label{eq:ball-inclusion}
B_{\sigma d_x}(x)\subset B_{r_M}(M)\setminus M,
\end{equation} and 
\begin{equation}\label{eq:unif-common-direction}
\Big\langle \frac{y-z}{|y-z|},e_x\Big\rangle_{\R^n}\ge \frac12 \cos R(\alpha)>0
\quad\Foa y\in B_{\sigma d_x}(x),
\,z\in B_{(1+\sigma)d_y}(y)\cap M,
\end{equation}
where $R(\alpha)$ is given in \eqref{eq:geodesic-radius}.
 \end{corollary}
 \begin{proof}
Take $\epsilon\in (0,\epsilon_0)$ for $\epsilon_0=\epsilon_0(\alpha)$ from Lemma 
\ref{lem:set-directions}. 
For all $0<\sigma\le\frac{\epsilon}4<\frac{\epsilon_0}4<\frac14$ and 
$y\in B_{\sigma d_x}(x)$ one has 
$$
\textstyle
0< (1-\sigma)d_x<d_y=\dist(y,M)<(1+\sigma)d_x<\frac{1+\sigma}{1+\epsilon}r_M<r_M,
$$ 
which implies \eqref{eq:ball-inclusion}.
Moreover, 
\begin{equation}\label{eq:balls-inclusion}
B_{(1+\sigma)d_y}(y)\subset B_{[(1+\sigma)^2+\sigma]d_x}(x)
\subset B_{(1+4\sigma)d_x}(x)\subset B_{(1+\epsilon)d_x}(x).
\end{equation}
Estimate for $y\in B_{\sigma d_x}(x)$ and $z\in B_{(1+\sigma)d_y}(y)\cap M$
\begin{align*}
\textstyle
\big|\frac{x-z}{|x-z|}-
\frac{y-z}{|y-z|}\big|&
\textstyle
=\big|\frac{(x-z)|y-z|-(y-z)|x-z|}{|x-z|\cdot |y-z|}\big|
\textstyle
\le \frac1{|x-z|}\big(|(x-z)-(y-z)|+\big| |y-z|-|x-z|\big|\big)\\
&\textstyle
\le 2\frac{|x-y|}{|x-z|}<2\frac{\sigma d_x}{d_x}=2\sigma,
\end{align*}
because $|x-z|\ge d_x$ for all $z\in M$. Thus, by \eqref{eq:balls-inclusion} and 
\eqref{eq:common-direction} in Corollary \ref{cor:common-direction},
\begin{eqnarray*}
\textstyle
\big\langle\frac{y-z}{|y-z|},e_x\big\rangle_{\R^n}& =  &
\textstyle
\big\langle\frac{x-z}{|x-z|},e_x\big\rangle_{\R^n}- 
\big\langle\frac{x-z}{|x-z|}-\frac{y-z}{|y-z|},e_x\big\rangle_{\R^n}\\
 &\overset{\eqref{eq:balls-inclusion},\eqref{eq:common-direction}}{>} & 
 \textstyle
 \cos 
R(\alpha)-2\sigma\ge\frac12\cos R(\alpha)>0
\end{eqnarray*}
if $\sigma\le\sigma_0(\alpha,\epsilon):=\frac14 \min\{\cos R(\alpha),\epsilon\}.$
 \end{proof}
 By means of a suitable covering and a partition of unity we can use the unit vectors
 $e_x\in\S^{n-1}$ of Corollary \ref{cor:loc-unif-common-direction} to construct a
 globally defined smooth vector field that near the set $M$ points away from $M$.
 \begin{proposition}[Vector field pointing away from $M$]
 \label{prop:vf}
For a non-empty,  closed and path-connected set $M\subset\R^n$ with
$\delta(M,r_M)<g_n$ at scale $r_M>0$ there exists a vectorfield
$V\in C^\infty(\R^n,\R^n)$ such that for all $y\in B_{r_M}(M)\setminus M$
\begin{equation}\label{eq:vf}
\Big\langle\frac{y-z}{|y-z|},V(y)\Big\rangle_{\R^n}\ge 1\quad
\Foa  z\in M\,\,\textnormal{with}\,\,
|z-y|=d_y=\dist(y,M).
\end{equation}
 \end{proposition}
 \begin{proof}
Let $\delta(M,r_M)=g(\alpha_M)<g_n$ and fix 
$\alpha:=\frac12(\alpha_M+\beta_n)\in (\alpha_M,\beta_n)$, where
$\beta_n$ is given in \eqref{eq:beta_n}. For this particular  $\alpha$ let
$\epsilon_0(\alpha)>0$ 
be the number of Corollary \ref{cor:loc-unif-common-direction}. For any $x\in U_0:=
B_{r_M}(M)\setminus M$ there exists $\varepsilon_x\in (0,\epsilon_0(\alpha))$ such
that $x\in U_{\varepsilon_x}:=B_{\frac{r_M}{1+\varepsilon_x}}(M)\setminus M$.
According to Corollary \ref{cor:loc-unif-common-direction} there exists $e_x\in\S^{n-1}$
and $\sigma_0=\sigma_0(\alpha,\varepsilon_x)>0$ such that \eqref{eq:unif-common-direction}
holds true for all $y\in B_{\sigma_0 d_x}(x)$ and $z\in B_{(1+\sigma_0)d_y}(y)\cap M$.

 For the covering $U_0=\bigcup_{x\in U_0}B_{\sigma_0(\alpha,\varepsilon_x)d_x}(x)$
 we choose\footnote{This is possible since each subset of $\R^n$ is a metric space
 with the inherited Euclidean metric, and therefore paracompact.}
 a locally finite refinement $\mathcal{W}:=\bigcup_{x\in Z}W(x)=U_0$,
 where for each $x\in Z\subset U_0$ the set 
 $W(x)\subset B_{\sigma_0(\alpha,\varepsilon_x)d_x}(x)$ is open and contains
 $x.$ Since $U_0\subset\R^n$ is Hausdorff there is a partition
 of unity subordinate to the locally finite open covering $\mathcal{W}$, i.e.,
 there are functions $(\xi_x)_{x\in Z}\subset C^\infty(\R^n)$ with $0\le\xi_x\le 1$,
 $\supp\xi_x\subset W(x)$, $\sum_{x\in Z}\xi_x(y)=1$ for all $y\in U_0$,
 and, finally, for each $y\in U_0$ there is an open neighborhood $Y(y)$
 of $y$ such
 that we have $\xi_x|_{Y(y)}\equiv 0$ for all but finitely many $x\in Z$. Setting
 $w_x:=\frac{2e_x}{\cos R(\alpha)}$, where $R(\alpha)$ is given by 
 \eqref{eq:geodesic-radius} for $\alpha=\frac12 (\alpha_M+\beta_n)$,
 we define 
 \begin{equation}\label{eq:def-vf}
\textstyle
V:\R^n\to\R^n,\quad y\mapsto \sum_{x\in Z}\xi_x(y)w_x,
 \end{equation}
 which is a smooth vector field since for each $y\in U_0$ the sum is a finite sum on 
 an open neighborhood $Y(y)$. 
 For $y\in\supp\xi_x\subset W(x)\subset
 B_{\sigma_0(\alpha,\varepsilon_x)d_x}(x)$ for some $x\in Z$, and for $z\in M$ with 
 $|z-y|=d_y$ one has by \eqref{eq:unif-common-direction} in Corollary 
 \ref{cor:loc-unif-common-direction}
 $$
 \textstyle
 \big\langle\frac{y-z}{|y-z|},w_x\big\rangle_{\R^n}=\frac2{\cos R(\alpha)}
 \big\langle\frac{y-z}{|y-z|},e_x\big\rangle_{\R^n}
 \overset{\eqref{eq:unif-common-direction}}{\ge}
1,
$$
because $|z-y|=d_y<(1+\sigma_0(\alpha,\varepsilon_x))d_y$ so that
$z\in B_{(1+\sigma_0(\alpha,\varepsilon_x))d_y}(y)\cap M$. We conclude with
$$
\textstyle
\big\langle\frac{y-z}{|y-z|},V(y)\big\rangle_{\R^n}=\sum_{x\in Z}\xi_x(y)
 \big\langle\frac{y-z}{|y-z|},w_x\big\rangle_{\R^n}\ge\sum_{x\in Z}\xi_x(y)=1
 $$
  for all $y\in U_0$.
 \end{proof}

\section{Pseudo-gradient flow for the distance function}\label{sec:pseudo}
The globally defined vector field $V$ constructed in Proposition \ref{prop:vf}
for the set $M\subset\R^n$ satisfying the local distortion bound induces
geometric flows on the complement $\R^n\setminus M$ increasing or decreasing the distance
to $M$. To construct a \emph{distance-increasing} flow $h_M^+:\R^n\setminus M\times [0,1]\to
\R^n\setminus M$ take a cut-off function $\varphi\in C^\infty(\R^n)$ with
\begin{equation}\label{eq:varphi}
\begin{cases}
\varphi\ge 0 & \ON\R^n,\\
\varphi\equiv 1 & \ON B_{\frac12{r_M}-\rho}(M),\\
\varphi\equiv 0 & \ON \R^n\setminus B_{\frac12{r_M}}(M)
\end{cases}
\end{equation}
for some $\rho\in (0,\frac12{r_M})$. Now, define 
\begin{equation}\label{eq:vtilde}
\tilde{V}:=\frac12 r_M\varphi\cdot V\in C^\infty(\R^n,\R^n),
\end{equation}
and consider the initial value problem
\begin{equation}\label{eq:awp}
\begin{cases}
\frac{d}{dt}h(x,t)=\tilde{V}(h(x,t))&\Fo x\in\R^n\setminus M,\,t\in\R,\\
\quad\! h(x,0)=x.
\end{cases}
\end{equation}
According to the Picard-Lindel\"of theorem \eqref{eq:awp} possesses a unique
smooth solution for all times $t\in\R$
with smooth dependence on the initial data.
Define $h_M^+\in C^\infty(\R^n\setminus M\times [0,1],\R^n)$ to be the 
restriction of that solution to the time interval $[0,1]$.
\begin{theorem}[Distance-increasing flow]\label{thm:dist-inc-flow}
Let $M\subset\R^n$ be a non-empty, closed and path-connected set with local distortion
$\delta(M,r_M)<g_n$ at scale $r_M>0$ for $g_n$ as defined in \eqref{eq:dekster-jung}.
Then for every $\rho\in (0,\frac12{r_M})$ there is a function $h_{M,\rho}^+\equiv
h_M^+
\in C^\infty(\R^n\setminus M\times [0,1],\R^n)$ with image $h_M^+(\R^n\setminus
M\times [0,1])\subset \R^n\setminus M$, satisfying
\begin{enumerate}
\item[\rm (1+)] $h_M^+(\cdot,0)=\Id_{\R^n\setminus M},$
\item[\rm (2+)] $h_M^+(\R^n\setminus M,1)\subset \R^n\setminus B_{\frac12{r_M}-\rho}(M)$,
\item[\rm (3+)] $h_M^+(\cdot,t)|_{\R^n\setminus B_{\frac12{r_M}}(M)}=\Id_{\R^n\setminus
B_{\frac12{r_M}}(M)}$ for all $t\in [0,1]$.
\end{enumerate}
Furthermore, the map $h_M^+(\cdot,t):\R^n\setminus M\to\R^n$ is an
embedding for every $t\in [0,1]$, and the function $t\mapsto\dist(h_M^+(x,t),M)$
is non-decreasing on $[0,1]$ for each $x\in\R^n\setminus M$.
\end{theorem}
\begin{proof}
We construct $h_M^+$ as above as the smooth and unique solution to \eqref{eq:awp}
restricted to $[0,1]$.
Uniqueness secures that $h_M^+(\cdot,t):\R^n\setminus M\to\R^n$ is an embedding for
each $t\in [0,1]$. The prescribed initial value in \eqref{eq:awp} guarantees condition
(1+). For $x\in\R^n\setminus B_{\frac12{r_M}}(M)$ the constant function $t\mapsto
h_M^+(x,t):=x$ is the unique solution to \eqref{eq:awp} since 
$\tilde{V}(h_M^+(x,t))=\tilde{V}(x)=\frac12{r_M}\varphi(x)V(x)=0$ because of
\eqref{eq:varphi}, which settles claim (3+). 

As soon as we have proven
that the function $t\mapsto
D(t):=\dist(h_M^+(x,t),M)$ is non-decreasing for each $x\in\R^n\setminus M$, we also know
that $h_M^+(\R^n\setminus M\times [0,1])\subset\R^n\setminus M$.  To prove this
monotonicity it suffices to consider $x\in B_{\frac12{r_M}}(M)\setminus M$ because of
of (3+). Notice that the image $H_x:=h_M^+(x,[0,1])\subset\R^n$ is compact, so that 
by means of (1+)
\begin{equation}\label{eq:Delta}
\Delta_x:=\max_{y\in H_x}\dist(y,M) > 0
\qquad\text{for any }x\in B_{\frac12r_M}(M)\setminus M.
\end{equation}
Therefore,
\begin{equation}\label{eq:Kx}\tag{$*$}
\textstyle\dist(y,M)=\inf_{z\in M}|y-z|=\min_{z\in K_x}|y-z|
=\dist(y,K_x)
\quad\Foa y\in H_x
\end{equation}
where we have set $K_x:=\overline{B_{2\Delta_x}(H_x)}\cap M$. Since $K_x$
is compact we may apply
\cite[Theorem~2.1(4)]{clarke_1975} and 
\cite[Exercise 9.13, p.~99]{clarke-etal_1998}
to infer that the Clarke gradient $\partial (\dist(y,M))$ of the $1$-Lipschitz
function $y\mapsto\dist(y,M)$ is given by
\begin{equation}\label{eq:clarke-gradient}
\textstyle
\partial (\dist(y,M))=\conv \big(\big\{\frac{y-z}{|y-z|}:z\in\Gamma_{K_x}(y)\big\}\big)
\quad\Foa y\in H_x\setminus M \stackrel{\eqref{eq:Delta}}=H_x,
\end{equation}
where $\conv(A)$ denotes the convex hull of a set $A\subset\R^n$, and
$\Gamma_A(y):=\{z\in A:\dist(y,A)=|y-z|\}$. (Notice that here,
$\Gamma_{K_x}(y)=\Gamma_M(y)
$ for all $y\in H_x\setminus M\stackrel{\eqref{eq:Delta}}=H_x$.\footnote{%
``$\subset$'' follows by $K_x\subset M$ and~\eqref{eq:Kx}.
``$\supset$'' follows by~\eqref{eq:Kx} and $\dist(y,M)\le\Delta_x$,
so any $z\in M$ with $\dist(y,M)=|y-z|$ satisfies $z\in \overline{B_{\Delta_x}(y)}
\subset\overline{B_{\Delta_x}(H_x)}$.})
Since $h_M^+$ is smooth, the function $D(t)=\dist(h_M^+(x,t),M)$ is
Lipschitz continuous on $[0,1]$, and we can apply the non-smooth
vectorial mean-value theorem \cite[Proposition 2.6.5]{clarke_1990} to deduce
for $0<t_1<t_2<1$ that
\begin{equation}\label{eq:mean-value}
D(t_2)-D(t_1)\in\conv\big(\partial D([t_1,t_2])(t_2-t_1)\big),
\end{equation}
which means that the difference $D(t_2)-D(t_1)$ is contained in the convex hull of the 
set $\{Z(t_2-t_1):Z\in\partial D(u)\,\,\textnormal{for some $u\in [t_1,t_2]$}\}$.
To obtain information about the Clarke gradient $\partial D(u)$ we apply
Clarke's non-smooth chain rule \cite[Theorem~2.3.9]{clarke_1990} which implies
for $u\in[t_1,t_2]$
\begin{equation}\label{eq:chain-rule}
\textstyle
\partial D(u)\subset\overline{\conv}\big(\big\{\langle\alpha,\zeta\rangle_{\R^n}:
\zeta=\frac{d}{dt}|_{t=u}h_M^+(x,t),\,\alpha\in\partial\big(\dist(
y,M)\big)
{},y=h_M^+(x,u)\big\}\big).
\end{equation}
With \eqref{eq:clarke-gradient} for $y:=h_M^+(x,u)\in H_x\setminus M$ and
$\alpha\in\partial(\dist(y,M))$ we can represent $\alpha$ as a convex combination
(where the number $J_u\in\N$ of summands depends on $u$)
$$
\textstyle
\alpha=\sum_{i=1}^{J_u}\lambda_i\frac{y-z_i}{|y-z_i|},\,\,\,
\lambda_i\in [0,1]\,\,\forall\,
i=1,\ldots,J_u,\,\,\,
\sum_{i=1}^{J_u} \lambda_i=1,
$$
where $z_i\in M$ satisfy $d_y=\dist(y,M)=|y-z_i|$ for all $i=1,\ldots,J_u.$
Consequently, by \eqref{eq:awp} and \eqref{eq:vf} in Proposition \ref{prop:vf}
\begin{align}\label{eq:alpha-times-zeta}
\langle\alpha,\zeta\rangle_{\R^n}&
\textstyle
=\sum_{i=1}^{J_u}\lambda_i\big\langle
\frac{y-z_i}{|y-z_i|},\frac{d}{dt}|_{t=u}h_M^+(x,t)\big\rangle_{\R^n}
\overset{\eqref{eq:awp}}{=}
\sum_{i=1}^{J_u}\lambda_i\big\langle\frac{y-z_i}{|y-z_i|},\tilde{V}(h_M^+(x,u))
\big\rangle_{\R^n}\notag\\
&\textstyle
=\frac12 r_M\varphi(y)\sum_{i=1}^{J_u}\lambda_i\big\langle\frac{y-z_i}{|y-z_i|},
{V}(y)\big\rangle_{\R^n}\overset{\eqref{eq:vf}}{\ge}\frac12 r_M\varphi(y)\ge 0,
\end{align}
if $y=h_M^+(x,u)\in B_{r_M}(M)\setminus M$, which is certainly the case
if $t_2\ll 1$ by continuity of the flow and our initial condition 
$h_M^+(x,0)=x\in B_{\frac12{r_M}}(M)\setminus M$. So, combining \eqref{eq:chain-rule}
with \eqref{eq:alpha-times-zeta} we obtain 
\begin{align}
\partial D(u) & 
 \subset \big[\tfrac12 r_M\varphi(h_M^+(x,u)),\infty\big),\label{eq:clarke-gradient-D}
\end{align}
which inserted into \eqref{eq:mean-value} yields
\begin{align*}
D(t_2)-D(t_1)&\in\conv\Big(\big\{Z\cdot (t_2-t_1):Z\in \big[\textstyle\frac12 r_M\varphi
(h_M^+(x,u)),\infty\big)\quad\textnormal{for some $u\in [t_1,t_2]$}\big\}\Big)\\
& \subset \big[\textstyle\frac12 r_M\min_{u\in [t_1,t_2]}\varphi(h_M^+(x,u))(t_2-t_1),
\infty
\big)\subset [0,\infty).
\end{align*}
Therefore, 
\begin{equation}\label{eq:D-diff}
D(t_2)-D(t_1)\ge \frac12 r_M\min_{u\in [t_1,t_2]}\varphi(h_M^+(x,u))(t_2-t_1)
\ge 0\qquad
\Foa 0<t_1<t_2<1,
\end{equation}
as long as $h_M^+(x,[0,t_2])\subset B_{r_M}(M)\setminus M$. 
This means that the distance $D(t)=\dist(h_M^+(x,t),M)$ is non-decreasing as long
as $h_M^+(x,t)\in B_{r_M}(M)\setminus M$, and as soon as
$D(t)=\frac12{r_M}$ it remains constant by means of the already established
property 
(3+) and uniqueness of the solution to \eqref{eq:awp}, 
so that $D(t_2)-D(t_1)\ge 0$ for all $0\le t_1<t_2\le 1$
by continuity of $D$.

In order to settle claim (2+) it suffices to assume that the starting point
$h_M^+(x,0)=x$ is contained in $B_{\frac12{r_M}}(M)\setminus M$ because of (3+).
We have to show that $D(1)=\dist(h_M^+(x,1),M)\ge\frac12{r_M}-\rho$.
If $D(t)\ge \frac12{r_M}-\rho$ for some $t\in [0,1)$ then $D(1)\ge D(t)\ge
\frac12{r_M}-\rho$ since $D(\cdot)$ was just shown to be non-decreasing
on $[0,1]$. If, on the other hand, $D(s)<\frac12{r_M}-\rho$ for all $s\in [0,1)$ then
by definition \eqref{eq:varphi} of $\varphi$ one has $\varphi(h_M^+(x,s))=1$
for all $s\in [0,1)$ (even for $s=1$ by continuity of $\varphi\circ h_M^+(x,\cdot)$),
and we can apply \eqref{eq:D-diff} to all $0<t_1<t_2<1$, for example
for $t_1:=\sigma$ and $t_2:=1-\sigma$ for any $\sigma\in (0,\frac12).$
Hence,
$$
\textstyle
D(1-\sigma)\overset{\eqref{eq:D-diff}}{\ge} D(\sigma)+\frac12 r_M(1-2\sigma),
$$
from which we deduce $D(1)\ge D(0)+\frac12 r_M>\frac12 r_M-\rho$ also in this
case, by sending $\sigma$ to zero.
\end{proof}
For future reference we draw the immediate conclusion, that having reached 
a certain
distance
from the set $M$ at time $t=1$ according to property (2+) of Theorem~
\ref{thm:dist-inc-flow}, means that we have also moved away from sets $A$
close to $M$
with respect to the Hausdorff distance $\dist_{\HH}$.
More precisely, we can state then following assertion.
\begin{corollary}[Complements of nearby sets under $h_M^+$]\label{cor:nearby-sets}
Let $M\subset\R^n$ be a non-empty, closed and path-connected set with
$\delta(M,r_M)<g_n$ at some scale $r_M>0$, and suppose a set $A\subset\R^n$ 
satisfies
\begin{equation}\label{eq:nearby-sets}
\dist_{\HH}(M,A)<\frac12{r_M} -\rho\quad\textnormal{for some $\rho\in (0,
\frac12{r_M}).$}
\end{equation}
Then the image of the function $h_M^+(\cdot,1)$ 
in Theorem~\ref{thm:dist-inc-flow} for this
$\rho$ satisfies the inclusion
\begin{equation}\label{eq:nearby-complements}
h_M^+(\R^n\setminus M,1)\subset\R^n\setminus A.
\end{equation}
\end{corollary}
\begin{proof}
Property (2+) in Theorem~\ref{thm:dist-inc-flow} implies by means of 
\eqref{eq:nearby-sets} for all $x\in\R^n\setminus M$
\[
\textstyle
\dist(h_M^+(x,1),A)\ge\dist(h_M^+(x,1),M)-\dist_{\HH}(M,A)\overset{\textnormal{(2+)}}{
\ge}\frac12{r_M} -\rho-\dist_{\HH}(M,A)\overset{\eqref{eq:nearby-sets}}{>} 0.\qedhere
\]
\end{proof}
In a similar fashion we can flow \emph{towards} $M$.
\begin{theorem}[Distance-decreasing flow]\label{thm:dist-dec-flow}
If $M\subset\R^n$ satisfies the assumptions of Theorem~\ref{thm:dist-inc-flow},
then for every $0<\rho <\delta < r_M$ there is a function $h_{M,\rho,\delta}^-
\equiv h_M^-\in C^\infty (\R^n\setminus M\times [0,1],\R^n)$ with 
$h_M^-(\R^n\setminus M\times [0,1])\subset\R^n\setminus M$ such that
\begin{enumerate}
\item[\rm (1-)] $h_M^-(\cdot,0)=\Id_{\R^n\setminus M},$
\item[\rm (2-)] $h_M^-(B_\delta(M)\setminus M,1)\subset B_{\rho}(M)\setminus M$,
\item[\rm (3-)] $h_M^-(\cdot,t)|_{(B_{\frac12{\rho}}(M)\setminus M)\cup(\R^n\setminus B_{r_M}(M))}=
\Id_{(B_{\frac12{\rho}}(M)\setminus M)\cup(\R^n\setminus B_{r_M}(M))}$
 for all $t\in [0,1]$.
\end{enumerate}
Moreover, the map $h_M^-(\cdot,t):\R^n\setminus M\to\R^n$ is an
embedding for every $t\in [0,1]$, and the function $t\mapsto\dist(h_M^-(x,t),M)$
is non-increasing on $[0,1]$ for each $x\in\R^n\setminus M$.
\end{theorem}
\begin{proof}
Choose a cut-off function $\psi\in C^\infty(\R^n)$ with $\psi\ge 0$ on
$\R^n$, $\psi\equiv 0$ on $B_{\frac12{\rho}}(M)\cup(\R^n\setminus B_{r_M}(M))$,
and $\psi\equiv 1$ on $B_\delta(M)\setminus B_\rho(M)$. We take the unique
smooth solution $h\in C^\infty((\R^n\setminus M)\times \R,\R^n)$ of 
\eqref{eq:awp} with the vector field $\tilde{V}$ replaced by $\tilde{W}:=-r_M\psi\cdot
V$, where $V\in C^\infty(\R^n,\R^n)$ is the vector field constructed in
Proposition \ref{prop:vf}. Define $h_M^-\in C^\infty(\R^n\setminus M\times [0,1],
\R^n)$ as the restriction of $h$ to the interval $[0,1]$. Again (1-) immediately
follows from the initial condition in \eqref{eq:awp}, and uniqueness of the solution
implies that $h_M^-(\cdot,t):\R^n\setminus M\to\R^n$ is an embedding for each 
$t\in [0,1]$. Property (3-) follows from the fact that $\psi$ vanishes on
$B_{\frac12{\rho}}(M)\cup(\R^n\setminus B_{r_M}(M))$ since the constant function 
$t\mapsto h_M^-(x,t):=x$ is the unique solution to \eqref{eq:awp} on regions
where the right-hand side of \eqref{eq:awp} vanishes. This already implies that
$h_M^-(\R^n\setminus M\times [0,1])\subset\R^n\setminus M.$

To show that $D(t):=\dist(h_M^-(x,t),M)$ is non-increasing on $[0,1]$ it suffices
to consider $x\in B_{r_M}(M)\setminus B_{\frac12{\rho}}(M)$ since for all other
$x\in\R^n\setminus M$ the distance $D(t)$ is constant on $[0,1]$ by virtue
of the already established property (3-). The characterization 
\eqref{eq:clarke-gradient}
of the Clarke gradient 
$\partial (\dist(y,M))$ for $y\in h_M^-(x,[0,1])\setminus M$
continues to hold true, as well as the vectorial mean-value inclusion 
\eqref{eq:mean-value} for $0<t_1<t_2<1$, and the inclusion \eqref{eq:chain-rule}
for $\partial D(u)$, $u\in [t_1,t_2]$. For $y:=h_M^-(x,u)\in h_M^-(x,[0,1])\setminus
M$ we use \eqref{eq:clarke-gradient} and the differential equation
\eqref{eq:awp} with $\tilde{W}$ instead of $\tilde{V}$ to conclude by means of
\eqref{eq:vf} for
$\alpha\in\partial\dist(y,M))$ and $\zeta=\frac{d}{dt}|_{t=u}h_M^-(x,t),$
\begin{align*}
\langle\alpha,\zeta\rangle_{\R^n}&
\textstyle
=\sum_{i=1}^{J_u}\lambda_i\big\langle
\frac{y-z_i}{|y-z_i|},\tilde{W}(y)\big\rangle_{\R^n}
 = -r_M\psi(y)\sum_{i=1}^{J_u}\lambda_i\big\langle
\frac{y-z_i}{|y-z_i|},V(y)\big\rangle_{\R^n}\le -r_M\psi(y)
\end{align*}
where $\lambda_i\in [0,1]$, $\sum_{i=1}^{J_u}
\lambda_i=1$, $z_i\in M$ such that $d_y=\dist(y,M)=|y-z_i|$
for all $i=1,\ldots,J_u$,
if $y=h_M^-(x,u)\in B_{r_M}(M)\setminus M$, which holds true for sufficiently
small $t_2\in (0,1)$ by continuity. So, we obtain for such $t_2$ and $u\in [t_1,t_2]$
\begin{align*}
\partial D(u)
& \subset \big(-\infty,-r_M\psi(h_M^-(x,u))\big].
\end{align*}
Inserting  this into \eqref{eq:mean-value} yields
\begin{equation}\label{eq:Dmonotone}
D(t_2)\le D(t_1)-r_M\min_{u\in [t_1,t_2]}\psi(h_M^-(x,u))(t_2-t_1)\le D(t_1)
\end{equation}
for all $0<t_1<t_2<1$ 
subject to $h_M^-(x,[0,t_2])\subset B_{r_M}(M)\setminus M$,
which proves that $D(\cdot)$ is non-increasing as long as the flow remains
in $B_{r_M}(M)\setminus M$. But as soon as 
$D(\cdot)$ reaches the value $\frac{\rho}2$
it
remains constant due to (3-) and uniqueness of the flow.

To finally establish (2-) we may
assume that 
\begin{equation}\label{eq:contraD}
\delta>D(0)\ge D(s)\ge\rho\quad\Foa
s\in [0,1),
\end{equation}
since if $D(s)<\rho$ for
some $s\in [0,1)$, then by the monotonicity we just proved we find $D(1)<\rho$ as well,
which settles (2-).
So, with \eqref{eq:contraD}
we are in the situation that we can apply \eqref{eq:Dmonotone}
for all $0<t_1<t_2<1$ and that $\psi(h_M^-(x,s))= 1$ for all $s\in [0,1)$.
Consequently, by \eqref{eq:Dmonotone} for $t_1:=\sigma,$ $t_2:=1-\sigma$, $\sigma\in (0,\frac12)$ we find $D(1-\sigma)\le D(\sigma)-r_M(1-2\sigma).$ Upon $\sigma\to 0$ we obtain
by continuity and \eqref{eq:contraD} the inequality
$
0\le D(1)\le D(0)-r_M <\delta-r_M<0$, a contradiction, so that the situation 
\eqref{eq:contraD}  can never occur.
\end{proof}

\section{Isomorphisms of fundamental groups}\label{sec:iso}
This section is devoted to the proof of Theorem~\ref{thm:isomorphic}. 
To start with, notice that  the mappings $r\mapsto\delta(M_i,r)$,
$i=1,2$, are non-decreasing, where $M_i$ are
the two non-empty,
closed and path-connected sets with path-connected complement 
$\R^n\setminus M_i$,
such that
\begin{equation}\label{eq:distortion-assumptions}
\delta(M_i,r_{M_i})< g_n\quad\Fo i=1,2, \quad\AND 
\dist_{\HH}(M_1,M_2)<\frac14
\min\{r_{M_1},r_{M_2}\},
\end{equation}
as requested in Theorem~\ref{thm:isomorphic}. 
So, we set
\begin{equation}\label{eq:r*}
r^*:=\min\{r_{M_1},r_{M_2}\},
\end{equation}
and apply Theorem~\ref{thm:dist-inc-flow} to the set $M:=M_1$
but with respect to the (possibly smaller) scale $r^*$ instead
of $r_{M_1}$, and with 
the parameter $\rho$ required in Theorem~\ref{thm:dist-inc-flow} 
chosen to satisfy
\begin{equation}\label{eq:rho-choice}
0<\rho<\frac14\cdot\Big(\,\frac14\min\{r_{M_1},r_{M_2}\}-\dist_{\HH}(M_1,M_2)\Big)=\frac14\cdot\Big(\frac14 r^*-\dist_{\HH}(M_1,M_2)\Big),
\end{equation}
to obtain the smooth function $h_1:=h_{M_1}^+:\R^n\setminus M_1\times [0,1]\to\R^n
\setminus  M_1$ with all the properties listed in that theorem.
Define the 
group homomorphism
$J:\pi_1(\R^n\setminus M_1)\to 
\pi_1(\R^n\setminus M_2) $ by means of
\begin{equation}\label{eq:def-J}
J([c]_1):=[h_1(\cdot,1)\circ c]_2,
\end{equation}
for any closed curve $c$ in $\R^n\setminus M_1$ representing the equivalence\footnote{We
choose the obvious notation $[\cdot]_i$ 
to indicate equivalence classes of closed loops 
in
the respective fundamental  group $\pi_1(\R^n\setminus M_i)$ for $i=1,2$.}
class $[c]_1\in\pi_1(\R^n\setminus M_1)$. Corollary \ref{cor:nearby-sets}
for $M:=M_1$ and $A:=M_2$ implies that
\begin{equation}\label{eq:respect-complements}
h_1(\R^n\setminus M_1,1)\subset \R^n\setminus M_2,
\end{equation}
because our assumption \eqref{eq:rho-choice} on $\rho$ clearly implies
$\dist_{\HH}(M_1,M_2)<\frac14 r^*<\frac12 r^*-\rho, $
that is, hypothesis \eqref{eq:nearby-sets} of Corollary \ref{cor:nearby-sets}
is satisfied for $r_M:=r^*$.
Inclusion \eqref{eq:respect-complements} shows that the mapping
$J$ defined in
\eqref{eq:def-J}  actually maps into $\pi_1(\R^n\setminus M_2)$, but we need to
verify that the definition does \emph{not} depend on the particular choice
of the loop $c$ in $\R^n\setminus M_1$. Indeed, for any other representative
loop $\tilde{c}$ in $\R^n\setminus M_1$ with $[\tilde{c}]_1=[c]_1$ there is,
by definition, a homotopy $\psi\in C^0(
\R/\Z\times [0,1],\R^n\setminus M_1)$ such that
$\psi(\cdot,0)=c(\cdot)$ and $\psi(\cdot,1)=\tilde{c}(\cdot)$. Then by 
\eqref{eq:respect-complements} the map $h_1(\cdot,1)\circ\psi$ is a homotopy in $\R^n\setminus
M_2$ connecting $h_1(\cdot,1)\circ\psi(\cdot,0)=h_1(c(\cdot),1)$ and
$h_1(\cdot,1)\circ\psi(\cdot,1)=h_1(\tilde{c}(\cdot),1)$, so that
$J([c]_1)=J([\tilde{c}]_1).$ Hence $J$ is well-defined.
\begin{lemma}[$J$ injective]\label{lem:injective}
The linear mapping $J:\pi_1(\R^n\setminus M_1)\to\pi_1(\R^n\setminus M_2)$
as defined in \eqref{eq:def-J} above is injective.
\end{lemma}
\begin{proof}
It suffices to show that for any null-homotopic loop in $\R^n\setminus M_2$ the 
preimage of its equivalence class under $J$ only contains null-homotopic 
loops in $\R^n\setminus M_1$. To that extent let $[c]_1\in\ker J\subset
\pi_1(\R^n\setminus M_1)$,
which means that there is a homotopy $\Psi\in C^0(
\R/\Z\times [0,1],\R^n\setminus
M_2)$ such that $\Psi(\cdot,0)=h_1(\cdot,1)\circ c(\cdot)$ and
$\Psi(\cdot,1)$ is constant. By means of Theorem~\ref{thm:dist-inc-flow}
for $M:=M_2$ (again with respect to the possibly smaller
scale $r^*$ instead of $r_{M_2}$)
we find the smooth function $h_2:=h_{M_2}^+ 
\in C^\infty((\R^n\setminus M_2)\times[0,1],\R^n\setminus M_2)$ 
with all the properties
listed there, as well as
\begin{equation}\label{eq:respect-complements2}
h_2(\R^n\setminus M_2,1)\subset \R^n\setminus M_1
\end{equation}
by means of Corollary \ref{cor:nearby-sets} applied to $M:=M_2$ and
$A:= M_1$. This corollary is applicable since our assumption \eqref{eq:rho-choice}
implies the hypothesis \eqref{eq:nearby-sets} similarly as in our
justification of \eqref{eq:respect-complements} above. Consequently,
the mapping $(s,t)\mapsto h_2(\Psi(s,t),1)$ is a homotopy in $\R^n\setminus M_1$.
By means of the monotonicity of $t\mapsto\dist(h_2(x,t),M_2)$ for any
$x\in\R^n\setminus M_2$ established in Theorem~\ref{thm:dist-inc-flow},
we infer with the help of the properties (1+) and (2+) therein, and 
\eqref{eq:rho-choice},
\begin{eqnarray*}
\dist(h_2(\cdot,t)\circ\Psi(s,0),M_1) & \ge &\dist(h_2(\cdot,t)\circ\Psi(s,0),M_2)
-\dist_{\HH}(M_1,M_2)\\
&\ge &\dist(h_2(\cdot,0)\circ\Psi(s,0),M_2)
-\dist_{\HH}(M_1,M_2)\\
&\overset{\textnormal{(1+)}}{=} &\dist(\Psi(s,0),M_2)
-\dist_{\HH}(M_1,M_2)\\
&\ge &\dist(\Psi(s,0),M_1)
-2\dist_{\HH}(M_1,M_2)\\
&=& \dist(h_1(\cdot,1)\circ c(s),M_1)-2\dist_{\HH}(M_1,M_2)\\
&\overset{\textnormal{(2+)},\eqref{eq:rho-choice}}{>} &
\textstyle\frac12 r^*-\rho-2\big(\frac14 r^*-4\rho\big)\\
& = & 7\rho >0\quad\Foa s,t\in [0,1].
\end{eqnarray*}
Therefore, the mapping $(s,t)\mapsto h_2(\Psi(s,0),t)$ is also a homotopy
in $\R^n\setminus M_1$, as well as $(s,t)\mapsto h_1(c(s),t)$ with $
c(s)=h_1(c(s),0)$ due to property (1+) in Theorem~\ref{thm:dist-inc-flow}
for $M:=M_1$. We now conclude with (1+) for $M:=M_2$:
\begin{align*}
[c]_1 & = [h_1(c,1)]_1=[\Psi(\cdot,0)]_1\overset{\textnormal{(1+)}}{=}
[h_2(\Psi(\cdot,0),0)]_1\\
& = [h_2(\Psi(\cdot,0),1)]_1=[h_2(\Psi(\cdot,1),1)]_1=[const.]_1,
\end{align*}
since $\Psi(\cdot,1)$ is constant. Here we have also used that $(s,t)\mapsto
h_2(\Psi(s,t),1)$ is a homotopy in $\R^n\setminus M_1$ by means of 
\eqref{eq:respect-complements2}. Thus, we have shown that $J$ is injective.
\end{proof}
\begin{lemma}[$J$ is surjective]\label{lem:surjective}
The injective linear mapping $J:\pi_1(\R^n\setminus M_1)\to\pi_1(\R^n\setminus M_2)$ as
defined in \eqref{eq:def-J} is also surjective.
\end{lemma}
\begin{proof}
For a given closed curve $\tilde{c}$ in $\R^n\setminus M_2$ representing
$[\tilde{c}]_2\in\pi_1(\R^n\setminus M_2)$ consider the smooth functions $h_i:=
h_{M_i}^+$ from Theorem~\ref{thm:dist-inc-flow} for $M:= M_i$, $i=1,2$, 
but with respect to the common smaller scale $r^*$ defined
in \eqref{eq:r*} instead of $r_{M_i}$,
and with $\rho$ as in \eqref{eq:rho-choice} with all the 
properties listed in that theorem, and satisfying \eqref{eq:respect-complements},
and \eqref{eq:respect-complements2}, respectively.
Then we can estimate, again using that $t\mapsto\dist(h_1(x,t),M_1)$ 
is non-decreasing, properties (1+) for  $h_1$ and (2+) for $h_2$ and
our smallness assumption \eqref{eq:rho-choice}, 
\begin{eqnarray}\label{eq:dist-h1h2}
\dist(h_1(\cdot,t)\circ h_2(\cdot,1)\circ\tilde{c},M_2) & \ge &
\dist(h_1(\cdot,t)\circ h_2(\cdot,1)\circ\tilde{c},M_1)-\dist_{\HH}(M_1,M_2)
\notag\\
&\ge &
\dist(h_1(\cdot,0)\circ h_2(\cdot,1)\circ\tilde{c},M_1)-\dist_{\HH}(M_1,M_2)\notag\\
& \overset{\textnormal{(1+)}}{=} &
\dist(h_2(\cdot,1)\circ\tilde{c},M_1)-\dist_{\HH}(M_1,M_2)\notag\\
&\ge &
\dist(h_2(\cdot,1)\circ\tilde{c},M_2)-2\dist_{\HH}(M_1,M_2)\\
&\overset{\textnormal{(2+)},\eqref{eq:rho-choice}}{>} &
\textstyle \big(\frac12 r^*-\rho\big)-2\big(\frac14 r^*-4\rho
\big) \notag\\
& = & 7\rho > 0,\notag
\end{eqnarray}
to infer that the mapping $(s,t)\mapsto h_1(h_2(\tilde{c}(s),1),t)$
is a homotopy in $\R^n\setminus M_2$. Therefore, by definition \eqref{eq:def-J}
of $J$ and property (1+) for $h_1$
$$
J\big([h_2(\cdot,1)\circ\tilde{c}]_1\big)
\overset{\eqref{eq:def-J}}{=}  [h_1(\cdot,1)\circ h_2(\cdot,1)\circ\tilde{c}]_2
 =  [h_1(\cdot,0)\circ h_2(\cdot,1)\circ\tilde{c}]_2
\overset{\textnormal{(1+)}}{=} 
[h_2(\cdot,1)\circ\tilde{c}]_2=[\tilde{c}]_2,
$$
where for the last equality we used the fact that also $(s,t)\mapsto
h_2(\tilde{c}(s),t)$ is a homotopy in $\R^n\setminus M_2$.
\end{proof}
To investigate peripheral subgroups of the fundamental groups
we first use the distance-decreasing flow of Theorem~\ref{thm:dist-dec-flow}
to derive a quantified sufficient condition for loops $c\in \R^n\setminus M$
to generate elements $[c]$ in the peripheral subgroup of $\pi_1(\R^n\setminus M)$.
\begin{proposition}[Elements of peripheral subgroups]
\label{prop:elements-peripheral}
Let $M\subset\R^n$ be a non-empty, closed and path-connected set with path-connected
complement $\R^n\setminus M$, such that $\delta(M,r_M)<g_n$ for some positive
scale $r_M$. Then the equivalence class $[c]$ of any loop $c:\R/\Z\to
B_{r_M}(M)\setminus M$ is an element of the peripheral subgroup
$P_M$ of $\pi_1(\R^n\setminus M)$.
\begin{proof}
Fix a loop $c\in C^0(\R/\Z,\R^n)$ with image $c(\R/\Z)\subset B_{r_M}(M)\setminus M$. We
need to show that for arbitrary small $\epsilon >0$, the closed curve $c$ is homotopic
to a loop $\tilde{c}$ in $B_\epsilon(M)\setminus M$. If $\epsilon> 
\delta:=\max_{x\in c(\R/\Z)}\dist(x,M)$
the curve
$c$ itself is the desired loop, so we may restrict to $\epsilon \in 
(0,\delta]\subset (0,r_M)$.
Apply Theorem~\ref{thm:dist-dec-flow} for $\rho:=\epsilon$ and
$\delta:=\max_{x\in c(\R/\Z)}\dist(x,M)<r_M$ (assuming w.l.o.g.\@ that 
$\epsilon<\delta$, otherwise increase $\delta$ slightly),
to obtain a smooth function
$h:=h_M^-:\R^n\setminus M\times [0,1]\to\R^n\setminus M$ with all the properties
listed in that theorem. Then the map $(s,t)\mapsto h(c(s),t)
=h(\cdot,t)\circ c(s)$ defines a homotopy in $\R^n\setminus M$,
connecting the curve $c(\cdot)
=h(\cdot,0)\circ c(\cdot)$ with the loop $\tilde{c}(\cdot):=h(c(\cdot),1).$ Property
(2-) of $h$ implies that $\tilde{c}(\R/\Z)\subset B_\epsilon(M).$
\end{proof}
\end{proposition}
Now we are in the position to give the full 
proof of our first main result.\\
{\it Proof of Theorem~\ref{thm:isomorphic}.}\,
The mapping $J:\pi_1(\R^n\setminus M_1)\to\pi_1(\R^n\setminus M_2)$ defined in
\eqref{eq:def-J} is an isomorphism as shown in Lemmas \ref{lem:injective} and
\ref{lem:surjective}. It remains to be checked that $J$ restricted to the
peripheral subgroup $P_{M_1}\subset\pi_1(\R^n\setminus M_1) $ maps $P_{M_1}$
isomorphically onto the peripheral subgroup $P_{M_2}$ of $\pi_1(\R^n\setminus M_2)$.
In the light of Lemma~\ref{lem:injective} is it sufficient 
to show $J(P_{M_1}) = P_{M_2}$.\par
``$\subset$'':
For $p\in P_{M_1}$ take a representative loop $c$ in $\R^n\setminus M_1$, i.e., 
$[c]_1=p$, with 
\begin{equation}\label{eq:close-loop}
\textstyle
\sup_{s\in\R/\Z}\dist\big(c(s),M_1\big)< \frac12 \min\{r_{M_1},r_{M_2}\}=\frac12 r^*,
\end{equation}
where $r^*$ was defined in \eqref{eq:r*}.
Let $h_1:=h_{M_1}^+:\R^n\setminus M_1\to\R^n\setminus M_1$
be the smooth function\footnote{This function was the one used for
the definition of the isomorphism $J$ in \eqref{eq:def-J}.}
of Theorem~\ref{thm:dist-inc-flow} for $M:=M_1,
$ but using the scale $r^*$   
instead of $r_{M_1}$ and $\rho$ as in \eqref{eq:rho-choice}.
Then, properties (2+) and (3+) of that theorem imply
\begin{equation}\label{eq:peripheral-ineq}
\textstyle
0<\frac14{r^*}<\frac12{r^*}-\rho\overset{\textnormal{(2+)}}{\le}\dist(h_1(c(s),1),M_1)\le
\frac12{r^*}\quad\Foa s\in \R/\Z,
\end{equation}
since if there were some $s\in\R/\Z$ such that $\dist (h_1(c(s),1),M_1)>
\frac12{r^*}$ then by continuity of $h_1$ we could find some $t^*\in (0,1)$
such that 
$$
\textstyle
\frac12{r^*}=\dist(h_1(c(s),t^*),M_1)\overset{\textnormal{(3+)}}{=}
\dist(h_1(c(s),1),M_1)>\frac12{r^*},
$$
contradiction. Moreover, by Corollary \ref{cor:nearby-sets} for $M:=M_1$ at scale
$r^*$ (instead of $r_{M_1}$), and $A:=M_2$, we have
\begin{equation}\label{eq:respect-complements3}
h_1(\R^n\setminus M_1,1)\subset\R^n\setminus M_2,
\end{equation}
since by means of \eqref{eq:rho-choice},
$\dist_{\HH}(M_1,M_2)<\frac14 r^*-4\rho<\frac12 r^*-\rho$,
so that hypothesis \eqref{eq:nearby-sets} of Corollary \ref{cor:nearby-sets}
is satisfied. In particular, $h_1(c(\R/\Z),1)\subset\R^n\setminus M_2$ so that
by \eqref{eq:peripheral-ineq} and hypothesis \eqref{eq:hausdorff}
\begin{eqnarray*}
0  < \dist(h_1(c(s),1),M_2) &\le &\dist(h_1(c(s),1),M_1)+\dist_{\HH}(M_1,M_2)\\
&\overset{\eqref{eq:peripheral-ineq},\eqref{eq:hausdorff}}{<} &
\textstyle
\frac12{r^*}+\frac14 \min\{r_{M_1},r_{M_2}\}
=\frac34 \min\{r_{M_1},r_{M_2}\}< r_{M_2}
\end{eqnarray*}
for all $s\in\R/\Z$.
Therefore, $h_1(c(\R/\Z),1)\subset B_{r_{M_2}}(M_2)\setminus M_2$. 

By definition \eqref{eq:def-J} of the isomorophism $J$ we find
$$
J(p)=J([c]_1)\overset{\eqref{eq:def-J}}{=}[h_1(\cdot,1)\circ c]_2\in P_{M_2}
\subset\pi_1(\R^n\setminus M_2)
$$
by means of Proposition \ref{prop:elements-peripheral} applied to
$M:=M_2$ and to the loop $h_1(\cdot,1)\circ c$.

``$\supset$'':
We finally need to show that $P_{M_2}\subset J(P_{M_1})$. For $q\in P_{M_2}$
there is a representative loop $\tilde{c}:\R/\Z\to\R^n\setminus M_2$ with
$[\tilde{c}]_2=q$ and $\sup_{s\in\R/\Z}\dist(\tilde{c}(s),M_2)<\frac12 r^*$.
According to Theorem~\ref{thm:dist-inc-flow} applied to $M:=M_2$ and
the smaller scale $r^*\le r_{M_2}$ and to 
$\rho$ as in \eqref{eq:rho-choice}, we find
the smooth function $h_2:=h_{M_2}^+:\R^n\setminus M_2\to\R^n\setminus M_2$ with 
all the properties listed in that theorem but
at scale $r^*$ (instead of $r_{M_2}$). Using properties (2+) and (3+)
in the same way as for $h_1$ above we obtain
\begin{equation}\label{eq:dist-peripheral2}
\textstyle
0<\frac14{r^*} <\frac12{r^*} -\rho\overset{\textnormal{(2+)}}{\le}
\dist\big(h_2(\tilde{c}(s),1),M_2\big)
\le\frac12{r^*}\quad\Foa s\in\R/\Z.
\end{equation}
From Corollary \ref{cor:nearby-sets} we deduce
\begin{equation}\label{eq:respect-complements4}
h_2(\R^n\setminus M_2,1)\subset\R^n\setminus M_1.
\end{equation}
Consequently,  by \eqref{eq:dist-peripheral2} and hypothesis \eqref{eq:hausdorff},
\begin{eqnarray*}
0<\dist\big(h_2(\cdot,1)\circ\tilde{c},M_1\big) & \le &
\dist\big(h_2(\cdot,1)\circ\tilde{c},M_2\big)+
\dist_{\HH}(M_1,M_2)\\
& \overset{\eqref{eq:dist-peripheral2},\eqref{eq:hausdorff}}{<} &
\textstyle
\frac12{r^*}+\frac14{r^*}=\frac34 \min\{r_{M_1},r_{M_2}\}<r_{M_1};
\end{eqnarray*}
hence $h_2(\tilde{c}(\R/\Z),1)\subset B_{r_{M_1}}(M_1)\setminus M_1$, which leads
us to set $p:=[h_2(\cdot,1)\circ\tilde{c}]_1\in P_{M_1}\subset\pi_1(\R^n\setminus M_1)$
by Proposition \ref{prop:elements-peripheral}.
Exactly\footnote{The only prerequisite for the estimate \eqref{eq:dist-h1h2} was that
$\tilde{c}$ is a loop in $\R^n\setminus M_2$ which we have in the present
context as well.}
as in \eqref{eq:dist-h1h2} in the proof of Lemma \ref{lem:surjective}
we can show that the mapping $(s,t)\mapsto h_1(h_2(\tilde{c}(s),1),t)$ defines a homotopy in 
$\R^n\setminus M_2$,  
 connecting the loop 
$$
h_2(\cdot,1)\circ\tilde{c}\overset{\textnormal{(1+)}}{=}h_1(h_2(\cdot,1)\circ\tilde{c},0)
$$
with $h_1(h_2(\cdot,1)\circ\tilde{c},1)$ so that by definition \eqref{eq:def-J} of $J$,
\begin{align*}
J(p)& = J([h_2(\cdot,1)\circ\tilde{c}]_1)\overset{\eqref{eq:def-J}}{=}
[h_1(\cdot,1)\circ h_2(\cdot,1)\circ\tilde{c}]_2\\
& = [h_2(\cdot,1)\circ\tilde{c}]_2=[\tilde{c}]_2=q,
\end{align*}
as desired. For the last equality we have used the fact that
$(s,t)\mapsto h_2(\tilde{c}(s),t)$ is a homotopy in $\R^n\setminus M_2$
connecting $\tilde{c}=h_2(\tilde{c},0)$ with $h_2(\cdot,1)\circ\tilde{c}$.
\hfill $\Box$

\section{Stability of knot equivalence}\label{sec:stability}
The equivalence of two tame knots $\g_1,\g_2\in C^0(\R/\Z,\R^3)$
with local distortion below $g_3$ (or even below
$\frac{\pi}{\sqrt{8}}$) at some positive scales, such that the knots 
are sufficiently close in 
terms of these scales, was stated in Corollary \ref{cor:equiv-knots}. This was a direct
consequence of Theorem~\ref{thm:isomorphic} applied to the knots' images $M_i:=
\g_i(\R/\Z)$, $i=1,2,$ in combination with Theorem~\ref{thm:gordon-luecke}. 
We are going to apply Corollary \ref{cor:equiv-knots} now
to a knot and a small Lipschitz perturbation of that knot to
prove the Lipschitz stability of knot equivalence, i.e., 
Corollary \ref{cor:lipschitz-stability}. For that we specify the local distortion
at some scale $r>0$ as
defined in \eqref{eq:local-distortion} in Definition \ref{def:local-distortion}
to the image $M:=\eta(\R/\Z)$ of a knot $\eta\in C^0(\R/\Z,\R^n)$ by writing
\begin{equation}\label{eq:local-distortion-knots}
\delta(\eta,r)\equiv\delta(\eta(\R/\Z),r) 
=\sup_{u,v\in\R/\Z, u\not= v\atop 0<|\eta(u)-\eta(v)|\le 2r}
\frac{d_\eta\big(\eta(u),\eta(v)\big)}{|\eta(u)-\eta(v)|},
\end{equation}
where the intrinsic distance $d_\eta(\eta(u),\eta(v))$ reduces to the periodic 
parameter distance $|u-v|_{\R/\Z}$ in case of arclength parametrized curves $\eta$.

\begin{proof}[Proof of Corollary {\ref{cor:lipschitz-stability}}]
By assumption there is a scale $r_\g\in(0,\tfrac12)$ such that $\delta(\g,r_\g)< g_3$ 
(or even
less than $\frac{\pi}{\sqrt{ 8 }}$.) Our aim is to show that for a given
function $h\in C^{0,1}(\R/\Z,\R^3)$ the perturbed
curve $\g_\epsilon=\g+\epsilon h$ has also local distortion
$\delta(\g_\epsilon,\hat{r}_{\g})<g_3 $ at
a (possibly smaller) scale $\hat{r}_\g>0$ independent of $0<|\epsilon|\ll 1$. Keeping
that scale fixed we can then take $|\epsilon| $ so small 
(depending on~$h$) that the $L^\infty$-distance
is controlled in terms of the two scales $r_\g$ and $\hat{r}_\g$.

To that extent consider $x,y\in [0,1]$ with $x<y$ 
to estimate
\begin{equation}\label{eq:euclid-eps}
|\g_\epsilon(x)-\g_\epsilon(y)|=
|\g(x)-\g(y)+\epsilon(h(x)-h(y))|\ge |\g(x)-\g(y)|-|\epsilon|\|h'\|_{L^\infty}(y-x).
\end{equation}
The intrinsic distance $d_{\g_\epsilon}(\g_\epsilon(x),\g_\epsilon(y))\equiv
d_{\g_\epsilon}(x,y)$ on $\g_\epsilon$ may be compared
to $d_\g(\g(x),\g(y))\equiv d_\g(x,y)$ on $\g$ by means of
\begin{align}\label{eq:intrinsic-eps-prep}
\textstyle
\big|\int_x^y |\g_\epsilon'(u)|du & -
\textstyle
\int_x^y |\g'(u)|du\big|=
\big|\int_x^y \big(|\g_\epsilon'(u)|-|\g'(u)|\big)du\big|\le\int_x^y
\big| |\g_\epsilon'(u)|-|\g'(u)|\big|du\notag\\
&\textstyle
\le\int_x^y |\g_\epsilon'(u)-\g'(u)|du=
\int_x^y|\epsilon h'(u)|du\le |\epsilon|\|h'\|_{L^\infty}(y-x),
\end{align}
whence
\begin{equation}\label{eq:intrinsic-eps}
\textstyle
d_{\g_\epsilon}(x,y)\le\int_x^y|\g_\epsilon'(u)|du
\overset{\eqref{eq:intrinsic-eps-prep}}{\le}
\int_x^y|\g'(u)|du+|\epsilon|\|h'\|_{L^\infty}(y-x)
\le(y-x)(1+|\epsilon|\|h'\|_{L^\infty}).
\end{equation}
Without loss of generality we may assume that 
$0< y-x\le\tfrac12$, so
$d_\g(x,y) = y-x$.
Combining \eqref{eq:euclid-eps} with \eqref{eq:intrinsic-eps} we deduce supposing
 $|\g(x)-\g(y)|<2r_\g$
\begin{align}
\begin{split}\label{eq:distortion-eps}
\textstyle\frac{d_{\g_\epsilon}(x,y)}{|\g_\epsilon(x)-\g_\epsilon(y)|}
&  \overset{\eqref{eq:intrinsic-eps},\eqref{eq:euclid-eps}}{\le}
\textstyle
\frac{d_{\g}(x,y)(1+|\epsilon|\|h'\|_{L^\infty})}{|\g(x)-\g(y)|-
|\epsilon|\|h'\|_{L^\infty}(y-x)}=
 \frac{d_{\g}(x,y)(1+|\epsilon|\|h'\|_{L^\infty})}{|\g(x)-\g(y)|
 \big(1-|\epsilon|\|h'\|_{L^\infty}\frac{d_\g(x,y)}{|\g(x)-\g(y)|}\big)} \\
 & \le \textstyle
 \delta(\g,r_\g)\cdot\frac{1+|\epsilon|\|h'\|_{L^\infty}}{
 1-|\epsilon|\|h'\|_{L^\infty}\delta(\g,r_\g)}<g_3,
\end{split}
\end{align}
if
\begin{equation}\label{eq:eps2}
\textstyle
|\epsilon|\le\epsilon_0:=\min\big\{\frac{g_3-\delta(\g,r_\g)}{\delta(\g,r_\g)(1+g_3)\|h'\|_{L^\infty}+1},\frac{r_\g}{\|h'\|_{L^\infty}+1}\big\}>0.
\end{equation}
Consequently, if $|\g_\epsilon(x)-\g_\epsilon(y)|\le2\hat r_\g$ where $\hat r_\g=\tfrac12 r_\g$
we infer
$|\g(y)-\g(x)|\le 2r_\g$, so $\delta(\g_\epsilon,\hat{r}_\g)<g_3$
for all $|\epsilon|\le\epsilon_0$
due to \eqref{eq:distortion-eps} and
\eqref{eq:eps2}.
It is important to notice that the scale
$\hat{r}_\g$ does \emph{not} depend on $\epsilon\in [-\epsilon_0,\epsilon_0]$.
With $\|\g-\g_\epsilon\|_{L^\infty}=|\epsilon|\|h\|_{L^\infty}$ this allows
us to choose
\begin{equation}\label{eq:final-epsilon-choice}
\textstyle
\epsilon_{\g,h}:=\min\Big\{\epsilon_0,\frac{\hat{r}_\g}{8\|h \|_{
L^\infty}+1}\Big\}
<\frac14 \min\{r_\g,\hat{r}_\g\}
\end{equation}
to also satisfy the hypothesis on the $L^\infty$-distance in Corollary 
\ref{cor:equiv-knots} to conclude that $\g\sim\g_\epsilon$ for all
$|\epsilon|\le\epsilon_{\g,h}.$ In addition, $\essinf_{\R/\Z}|\g_\epsilon'|
\ge\essinf_{\R/\Z}|\g'|-|\epsilon|\|h'\|_{L^\infty} > \frac12$
for all $|\epsilon|\le\epsilon_{\g,h}$, since $\epsilon_{\g,h}\le\epsilon_0\le
\tfrac12(\|h'\|_{L^\infty}+1)^{-1}$.
\end{proof}

Before treating the stability of knot equivalence in the setting of
fractional Sobolev regularity let us point out that bounding  the
bilipschitz constant\footnote{\label{foot:bilip}In \cite{freedman-etal_1994} a bilipschitz constant
is defined as any number $L>0$ such that $L^{-1}d_\gamma(u,v)|\le|\gamma(u)-\gamma(v)|\le Ld_\gamma(u,v)$
holds for all $u,v\in\R/\Z$.
If $\gamma$ is parametrized by arclength then the infimum of those $L$
is $\bilip_I(\g)^{-1}$.}
\begin{equation}\label{eq:def-bilip}
\bilip_I(\g):=\inf_{\substack{u,v\in I\\u\ne v}}\frac{|\g(u)-\g(v)|}{|u-v|_{\R/\Z}}
\end{equation}
on an interval $I\subset\R$ from below, is closely related to
controlling the local distortion $\delta(\g,\cdot)$ at some
scale related to $I$, if the absolutely continuous parametrization $\g$
has positive
speed $v_\g=\essinf_{\R/\Z}|\g'|$, i.e., if $\g$ is an immersed curve.

The seminorm $\lfloor \g'\rfloor $ defining fractional Sobolev spaces allows us to bound
$\bilip(\g)$ as  in the following lemma, which is essentially contained
in \cite[Lemma 2.1]{blatt_2012a}; see also \cite[Lemma 2.1]{blatt-etal_2024}.

\begin{lemma}\label{lem:seminorm-bilip}
Let $I\subset\R$ be an interval and $\g:I\to\R^n$ an absolutely continuous
immersed curve with positive speed $v_\g=\essinf_{I}|\g'|$. 
Then one has for all $s,t\in I$ with $s<t$
\begin{equation}\label{eq:seminorm-bilip}
|\g(t)-\g(s)|^2\ge\Big\{v_\g^2-\frac12\int_s^t\int_s^t\frac{|\g'(u)-\g'(v)|^2}{|u-v|^2}dudv\Big\}|t-s|^2.
\end{equation}
\end{lemma}
As an immediate consequence one deduces
\begin{corollary}[Fractional seminorm controls bilipschitz constant]
\label{cor:seminorm-bilip}
If $\g\in W^{\frac32, 2}(I,\R^n)$ is arclength parametrized, then
\begin{equation}\label{eq:seminorm-bilip2}
|\g(t)-\g(s)|^2\ge\big(1-\frac12 \lfloor\g'\rfloor^2_{\frac12,2,I}\big)|t-s|^2\quad\Foa
s,t\in I.
\end{equation}
\end{corollary}
\begin{proof}[Proof of Lemma {\ref{lem:seminorm-bilip}}]
For parameters $s,t\in I$ with $s<t$ estimate
\begin{gather*}
|\g(s) -\g(t)|^2  
\textstyle
= \int_s^t\int_s^t\langle\g'(\sigma),\g'(\tau)\rangle_{\R^n}d\sigma 
d\tau
 \textstyle
= \frac12  \int_s^t\int_s^t \big(
|\g'(\tau)|^2+
|\g'(\sigma)|^2-|\g'(\tau)-\g'(\sigma)|^2\big)d\sigma d\tau\\
\textstyle
\ge v_\g^2|t-s|^2-\frac12\int_s^t\int_s^t|\g'(\tau)-\g'(\sigma)|^2d\sigma d\tau
\ge\textstyle
|t-s|^2\big\{v_\g^2-\frac12\int_s^t\int_s^t
\frac{|\g'(\tau)-\g'(\sigma)|^2}{|\tau-\sigma|^2}d\sigma d\tau\big\}.\qedhere
\end{gather*}
\end{proof}

We observe in \eqref{eq:seminorm-bilip2} that a small fractional seminorm
yields a bilipschitz  estimate, which is the key ingredient for the proof 
of the fractional stability of knot equivalence.

\noindent
{\it Proof of Corollary \ref{cor:fractional-stability}.}\,
Due to the absolute continuity of the double integral defining the
fractional seminorm there is a radius $\rho=\rho(\g)\in (0,\frac14]$  such
that $\sup_{x\in\R/\Z}\lfloor \g'\rfloor_{\frac12,2,B_\rho(x)} <\frac1{2\sqrt{2}}.$
Therefore,
\begin{equation}\label{eq:eta-seminorm}
\textstyle
\sup_{x\in\R/\Z}\lfloor \eta' \rfloor_{\frac12,2,B_\rho(x)}
\le
\sup_{x\in\R/\Z}\lfloor \g' \rfloor_{\frac12,2,B_\rho(x)}+\lfloor
\eta'-\g' \rfloor_{\frac12,2} < \frac1{\sqrt{2}}
\end{equation}
for all $\eta \in B_{\frac1{2\sqrt{2}}}(\g)\subset W^{\frac32,2}(\R/\Z,\R^3)$. By
means of Corollary \ref{cor:seminorm-bilip} we infer
\begin{equation}\label{eq:bilip-gamma-eta}
\textstyle
\min\big\{|\g(t)-\g(s)|,|\eta(t)-\eta(s)|\big\}>\frac{\sqrt{3}}2 
|t-s|\ge\frac{\sqrt{3}}2 d_\g(t,s)
=\frac{\sqrt{3}}2 d_\eta(t,s)
\end{equation}
for all arclength parametrized curves $\eta\in B_{\frac1{2\sqrt{2}}}(\g)\subset
W^{\frac32,2}(\R/\Z,\R^3)$ and all parameters $s,t\in\R/\Z$ with
intrinsic distance $d_\g(t,s)\equiv d_\g(\g(s),\g(t))=|t-s|_{\R/\Z}=
d_\eta(\eta(s),\eta(t))\equiv d_\eta(t,s)<2\rho.$

On the other hand,
$$
\textstyle
\sigma_\g:=
\inf\{|\g(t)-\g(s)|:d_\g(s,t)=|t-s|_{\R/\Z}\ge 2\rho\}\in (0,\frac12].
$$
By Sobolev embedding there is a constant $C_E$ such that $\|\zeta\|_{L^\infty}
\le C_E\|\zeta\|_{W^{\frac32,2}}$ for all $\zeta\in W^{\frac32,2}(\R/\Z,\R^3)$,
so that
\begin{equation}\label{eq:infty-distance-eta-gamma}
\textstyle
\|\eta-\gamma\|_{L^\infty}< \frac{\sigma_\g}{16}\quad\Foa\eta\in B_{\epsilon_\g}(\g)
\subset W^{\frac32,2}(\R/\Z,\R^3),
\end{equation}
where we have set $\epsilon_\g:=\min\{\frac1{2\sqrt{2}},\frac{\sigma_\g}{16C_E}\}$.
Therefore, for all $s,t\in\R/\Z$ with $d_\g(s,t)=d_\eta(s,t)
\ge 2\rho$ one has
\begin{equation}\label{eq:euclid-eta-gamma-global}
\textstyle
\min\big\{|\g(t)-\g(s)|,|\eta(t)-\eta(s)|\big\}>\frac{7}{8}\sigma_\g
>\frac{1}{2}\sigma_\g.
\end{equation}
Defining the common scale $r_{\g}:=\frac{\sigma_\g}4$ we find in view of 
\eqref{eq:bilip-gamma-eta} and \eqref{eq:euclid-eta-gamma-global}
\begin{equation}\label{eq:max-eta-gamma-distortion}
\textstyle
\max\big\{\delta(\g,r_\g),\delta(\eta,r_\g)\big\}\le\frac{2}{\sqrt{3}}<g_3
\end{equation}
for all arclength parametrized curves $\eta\in B_{\epsilon_\g}(\g)\subset
W^{\frac32,2}(\R/\Z,\R^3)$, since 
$g_3=\frac{\sqrt{32}}{\sqrt{27}}\frac{\pi}{\sqrt{8}}\in 
(\frac2{\sqrt{3}},\frac{\pi}2)$.
Inequality \eqref{eq:infty-distance-eta-gamma} allows us to apply
Corollary \ref{cor:equiv-knots} to conclude. Notice that we also used
the fact that arclength parametrized $W^{\frac32,2}$-knots have finite
M\"obius energy according to \cite[Theorem~1.1]{blatt_2012a}, and
are therefore tame \cite[Theorem~4.1]{freedman-etal_1994}, which is needed
in Corollary \ref{cor:equiv-knots}.\hfill $\Box$

\section{Harmonic substitution}\label{sec:harmonic}
In this section we study the effects on the local distortion of
a curve if we replace some subarcs of
a curve by straight segments, since this is
what we will do later to adequately modify minimizing sequences  for the
M\"obius energy. By Blatt's characterization \cite[Theorem~1.1]{blatt_2012a}
of the corresponding
energy space we can assume  arclength parametrized  curves of class 
$W^{\frac32,2}(\R/\Z,\R^n)$, but keep in mind 
that such a replacement by straight
segments generally destroys this fractional Sobolev regularity.

We begin with a result that basically reflects the embedding of $W^{\frac32,2}$
into the space of bounded mean oscillation (BMO). Denote the integral mean of a function
$f$  over a set $A$ by 
$$
(f)_A\equiv\mvint_A f(z)\,dz := \frac{1}{|A|}\int_A f(z)\,dz,
$$
and abbreviate the annular
region 
\begin{equation}\label{eq:annular}
A_{r,\theta}(x):= B_r(x)\setminus B_{\theta r}(x)\subset\R/\Z\quad\Fo \theta\in (0,1),
\,x\in\R/\Z,
\end{equation}
with measure 
\begin{equation}\label{eq:measure-annulus}
|A_{r,\theta}(x)|=2r(1-\theta).
\end{equation}
\begin{lemma}[Quantitative embedding $W^{\frac32,2}\hookrightarrow \,\textnormal{BMO}$]
\label{lem:quant-embed}
Assume that for an arclength parametrized curve $\eta\in W^{\frac32,2}(\R/\Z,\R^n)$
there is a parameter
$x\in\R/\Z$ and there exist constants $\theta,r\in (0,\frac12)$, such that
\begin{equation}\label{eq:frac-seminorm-theta}
\lfloor \eta'\rfloor^2_{\frac12,2,A_{r,\theta}(x)} < \theta.
\end{equation}
Then we have
\begin{equation}\label{eq:bmo1}
\mvint_{B_r(x)}|\eta'(z)-(\eta')_{B_r(x)}|^2\,dz < 8\theta,
\end{equation}
and if $\theta < \frac18 $ in addition, then $|(\eta')_{B_r(x)}|>0$ and
\begin{equation}\label{eq:bmo2}
\mvint_{B_r(x)}|\eta'(z)-\nu|^2\,dz < 32\theta,
\end{equation}
where 
\begin{equation}\label{eq:bmo3}
\nu\equiv \nu_{B_r(x)}
:=\frac{(\eta')_{B_r(x)}}{|(\eta')_{B_r(x)}|}\in\S^{n-1}.
\end{equation}
\end{lemma}
\begin{proof}
We start with  estimates for
the mean quadratic deviation on the annulus  $A:=
A_{r,\theta}(x)$
with \eqref{eq:measure-annulus}, Jensen's inequality, and \eqref{eq:frac-seminorm-theta}:
\begin{eqnarray}\label{eq:mean-deviation-annulus}
\textstyle\mvint_{A}|\eta'(z)\!-\!(\eta')_A|^2dz 
& \overset{\eqref{eq:measure-annulus}}{=} &
\textstyle\frac1{2r(1-\theta)}\int_A|\eta'(z)\!-\!\mvint_A\eta'(u)du|^2dz\notag
\overset{\eqref{eq:measure-annulus}}{\le}
\textstyle
\frac1{4r^2(1-\theta)^2}\int_A\int_A|\eta'(z)\!-\!\eta'(u)|^2dudz\notag\\
&\le &\textstyle
\frac1{(1-\theta)^2}\int_A\int_A\frac{|\eta'(z)-\eta'(u)|^2}{|z-u|^2}dudz
 \overset{\eqref{eq:frac-seminorm-theta}}{<} 
\textstyle\frac{\theta}{(1-\theta)^2} < 4\theta,
\end{eqnarray}
since $\theta\in (0,\frac12)$. With $|\eta'|=1$ a.e.\@ on $\R/\Z$ we have the rough
estimate $|\eta'(\cdot)-(\eta')_A|\le 2$ a.e., which implies for the mean deviation from
$(\eta')_A$ on the full ball $B:=B_r(x)$
\begin{eqnarray}\label{eq:deviation-annulus-ball}
\textstyle\mvint_B|\eta'(z)-(\eta')_A|^2dz & = &
\textstyle
\frac1{2r}\int_A|\eta'(z)-(\eta')_A|^2 dz + \frac1{2r}\int_{B_{\theta r}(x)}
|\eta'(z)-(\eta')_A|^2 dz\notag\\
& \overset{\eqref{eq:measure-annulus}}{=} &
\textstyle
\frac{2r(1-\theta)}{2r}\mvint_A
|\eta'(z)-(\eta')_A|^2 dz+\frac{2\theta r}{2r}\mvint_{B_{\theta r}(x)}
|\eta'(z)-(\eta')_A|^2 dz\notag\\
& \overset{\eqref{eq:mean-deviation-annulus}}{<} &
(1-\theta)4\theta + 4\theta < 8\theta.
\end{eqnarray}
This proves \eqref{eq:bmo1}
since the function $c\mapsto\mvint_B|\eta'(z)-c|^2dz$ is minimized
on $\R^n$ by the mean value $c^*:=(\eta')_B\in\R^n$.

If $\theta\in (0,\frac18)$, 
then take the integral mean of the inequality
$|(\eta')_B|\ge |\eta'(z)|-|\eta'(z)-(\eta')_B|$ over $B=B_r(x)$
to obtain with \eqref{eq:bmo1} and H\"older's inequality
$|(\eta')_B|>1-\sqrt{8\theta}>0$, so that the vector $\nu$ in
\eqref{eq:bmo3} is well-defined. With the help of 
$|b-\frac{b}{|b|}|^2=|\frac{b|b|-b}{|b|}|^2=||b|-1|^2$ for $b:=(\eta')_B\in\R^n\setminus
\{0\}$ we obtain 
\begin{align*}
\textstyle
\mvint_B|(\eta')_B-\nu|^2 dz& =
\textstyle
\mvint_B\big|1-|(\eta')_B|\big|^2 dz =\mvint_B\big||\eta'(z)|-|(\eta')_B|\big|^2 dz
\le 
\mvint_B|\eta'(z)-(\eta')_B|^2 dz,
&
\end{align*}
so that by means of \eqref{eq:bmo1}
\[
\textstyle
\mvint_B|\eta'(z)-\nu|^2dz \le 2\big\{\mvint_B|\eta'(z)-(\eta')_B|^2dz
+\mvint_B|(\eta')_B-\nu|^2dz\big\}\overset{\eqref{eq:bmo1}}{<} 2 
\{8\theta +8\theta\}=
32\theta.\qedhere
\]
\end{proof}
To find suitable subarcs to be replaced by straight segments on a given curve,
we pick up an idea of Semmes \cite{semmes_1991a} 
from his study of 
chord-arc surfaces with small constants. 

Define for a fixed parameter $x\in\R/\Z$ and constants $r\in (0,\frac12)$ and
$\theta\in (0,\frac1{20})$ the \emph{localized excess function} on the
ball $B:=B_r(x)$ as
\begin{equation}\label{eq:loc-excess}
e_B(z):=\big(  \eta'(z)-\nu \big)\chi_B(z)\quad\Fo z\in\R,
\end{equation}
where $\nu\equiv\nu_B=\nu_{B_r(x)}\in\S^{n-1}$ is the vector defined
in \eqref{eq:bmo3} of Lemma \ref{lem:quant-embed}, and $\chi_A$ denotes
the characteristic function of a set $A\subset\R/\Z$. Moreover, we consider the
Hardy--Littlewood maximal function of the excess,
\begin{equation}\label{eq:hardy-littlewood}
\big( Me_B\big)(y):=\sup_{\rho >0}\,\mvint_{B_\rho(y)}|e_B(z)|\,dz
\end{equation}
satisfying the weak-type inequality (see, e.g., \cite[Chapter 3, 
Theorem~1.1, p.~101]{stein-shakarchi_2005})
\begin{align*}
\big|\big\{y\in\R: (Me_B)(y)>t\big\} \big| & 
\textstyle
\le \frac3{t}\|e_B\|_{L^1(\R)}=
\frac3{t}\int_B|\eta'(z)-\nu|dz\\
&\textstyle
\le\frac{6r}t \big(\mvint_B |\eta'(z)-\nu|^2 dz\big)^{\frac12}
\overset{\eqref{eq:bmo2}}{<}\frac{6r}t \sqrt{32\theta}<36
\sqrt{\theta}\cdot\frac{r}t\quad
\Foa t>0,
\end{align*}
where we used H\"older's inequality and \eqref{eq:bmo2}. Choosing $t:=\theta^{\frac14}$
we obtain
\begin{align*}
\textstyle
\big|\big\{y\in\R:(Me_B)(y)\le\theta^{\frac14} \big\}\cap B_{\frac{r}8}(x\pm\frac{r}2)
\big|&\ge\textstyle
|B_{\frac{r}8}(x\pm \frac{r}2)|-\big|\big\{y\in\R: (Me_B)(y)>\theta^{\frac14}\big\} 
\big|\\
& >\frac{r}4\big(1-144\cdot\theta^{\frac14}\big) >0
\end{align*}
for 
\begin{equation}\label{eq:theta1}
0<\theta<\theta_1:=\frac1{144^4}<\frac18.
\end{equation}
Consequently, the sets
\begin{equation}\label{eq:G-set}
\textstyle
G_\pm(x,\theta,r):=\big\{y\in\R: (Me_{B_r(x)})(y)\le\theta^{\frac14}\big\}\cap
B_{\frac{r}8}(x\pm\frac{r}2)\subset B_r(x)
\end{equation}
of \emph{good} endpoints of replaceable subarcs near the two parameters $x\pm\frac{r}2$ 
have positive measure if $\theta\in (0,\theta_1)$. 
The next lemma clarifies the term ``good'' in this context, by stating
that difference quotients of the curve $\eta$ rooted
at points in the set $G_\pm(x,\theta,r)$
can be approximated well by the unit vector $\nu=\nu_{B_r(x)}\in\S^{n-1}.$
\begin{lemma}[Difference quotients close to $\nu$]
\label{lem:difference-quotients}
Assume that the arclength parametrized curve $\eta\in W^{\frac32,2}(\R/\Z,\R^n)$
satisfies
\begin{equation}\label{eq:fractional-seminorm-theta2}
\textstyle
\lfloor\eta' \rfloor^2_{\frac12,2,A}<\theta\quad\textnormal{for some
$x\in\R/\Z$, $\theta\in (0,\theta_1)$ and $r\in (0,\frac12)$,}
\end{equation}
where $\theta_1$ is defined in \eqref{eq:theta1}, and $A:=A_{r,\theta}(x)=
B_r(x)\setminus B_{\theta r}(x)$. Then one has for
$\nu:=\nu_{B_r(x)}$ defined in \eqref{eq:bmo3}
\begin{equation}\label{eq:difference-quotients}
1\ge\Big\langle\frac{\eta(y)-\eta(x_0)}{y-x_0},\nu\Big\rangle_{\R^n}\ge
1-2\theta^{\frac14}\quad\AND\quad
\Big|
\Pi_{(\R\nu)^\perp}\Big(\frac{\eta(y)-\eta(x_0)}{y-x_0}\Big)\Big|
\le 2\theta^\frac18
\end{equation}
for all $y\in B_r(x)$ and $x_0\in G_\pm(x,\theta,r)$ as defined
in \eqref{eq:G-set}. Here, $\Pi_V$ denotes the orthogonal projection of $\R^n$ onto
a linear subspace $V\subset\R^n$.
\end{lemma}
\begin{proof}
The very left inequality in \eqref{eq:difference-quotients} is obvious since $\|\eta'
\|_{L^\infty}=1$ by assumption, and $|\nu|=1$. The second inequality can be 
derived with the fundamental theorem of calculus as follows:
\begin{align*}
\textstyle
1-\big\langle\frac{\eta(y)-\eta(x_0)}{y-x_0},\nu\big\rangle_{\R^n} & =
\textstyle
\frac1{y-x_0}\int_{x_0}^y\langle \nu-\eta'(u),\nu\rangle_{\R^n}du\\
&
\textstyle
\le\frac1{|y-x_0|}\int_{[x_0,y]}|\nu-\eta'(u)|du=
\frac1{|y-x_0|}\int_{[x_0,y]}|e_{B_r(x)}(u)|du,
\end{align*}
since both $y$ and $x_0$ are contained in $B_r(x)$.
The last integral may be bounded from above by
$$
\textstyle
\frac1{|y-x_0|}\int_{B_{|y-x_0|}(x_0)}|e_{B_r(x)}(u)|du
=2\mvint_{B_{|y-x_0|}(x_0)}|e_{B_r(x)}(u)|du\le 2\big(Me_{B_r(x)}\big)(x_0)\le
2\theta^{\frac14},
$$
because $x_0\in G_\pm(x,\theta,r)$, which proves the first chain of
inequalities
in \eqref{eq:difference-quotients}. The last inequality
follows from the first via
$\textstyle\big|\Pi_{(\R\nu)^\perp}(\zeta)\big|^2=
\big|\zeta-\langle\zeta,\nu\rangle\nu\big|^2=
|\zeta|^2-\langle\zeta,\nu\rangle^2\le 1-\langle\zeta,\nu\rangle^2=
(1+\langle\zeta,\nu\rangle)(1-\langle\zeta,\nu\rangle)\le 
2(1-\langle\zeta,\nu\rangle)
$ for all $\zeta\in\overline{B_1(0)}\subset\R^n$
(where we abbreviated 
$\langle\cdot,\cdot\rangle_{\R^n}$ by $\langle\cdot,\cdot\rangle$).
\end{proof}
Given a curve $\eta:\R/\Z\to\R^n$ and $N$ distinct parameters $x_1,\ldots,x_N\in
\R/\Z$, as well as $\theta\in (0,\theta_1)$, 
\begin{equation}\label{eq:little-r}
0<r<\frac12\min\big\{|x_i-x_k|_{\R/\Z}: i\not= k,\,i,k\in\{1,\ldots,N\}\big\}\le \frac14,
\end{equation}
and good sets $G_\pm^ i:=G_\pm(x_i,\theta,r)$ as defined  in \eqref{eq:G-set} for
$i=1,\ldots,N$, containing the parameters $x^i_\pm\in G_\pm^i$,
we replace the $N$ subarcs $\eta([x_-^i,x_+^i])$ by straight segments to define
the \emph{modified curve}
\begin{equation}\label{eq:modified-curve}
\tilde{\eta}(z):=\begin{cases}
\eta(x_-^i)+\frac{z-x_-^i}{x_+^i-x_-^i}\big(\eta(x_+^i)-\eta(x_-^i)\big) & \Fo
z\in [x_-^i,x_+^i], i=1\ldots,N,\\
\eta(z) & \Fo z\in\R/\Z\setminus\bigcup_{i=1}^N[x_-^i,x_+^i].
\end{cases}
\end{equation}
We call this procedure a \emph{harmonic substitution} comparable to Courant's and 
Morrey's  method in the theory of classic two-dimensional minimal surfaces
\cite[p. 25]{courant_1977}, \cite[Theorem~4.3.2]{morrey_1966}. But in the present
context, even if $\eta$ were smooth, the modified curve $\tilde{\eta}$ may not
even be of class $W^{\frac32,2}$. Or, in terms of the M\"obius energy, $\emob(\tilde{\eta})$ may be infinite even if $\tilde{\eta}$ is embedded and if the original curve $\eta$
has finite energy. Nevertheless, $\tilde{\eta}$ is 
$L^\infty$-close to $\eta$, and
its distortion  quotient remains under control.
\begin{lemma}[Distortion under harmonic substitution]
\label{lem:harmonic-substitution}
Let $\eta\in W^{\frac32,2}(\R/\Z,\R^n)$ be an arclength parametrized curve satisfying
\begin{equation}\label{eq:fractional-seminorm-theta3}
\lfloor \eta'\rfloor^2_{\frac12,2,A_{r,\theta}(x_i)}<\theta
\end{equation}
for 
finitely 
many distinct parameters $x_1,\ldots,x_N\in\R/\Z$ with $r\in (0,\frac14)$
chosen as in \eqref{eq:little-r},
and $\theta\in (0,\theta_2)$ for
$\theta_2:=256^{-4}$, and let $\tilde{\eta}$ be the modified curve defined
by harmonic substitution as in \eqref{eq:modified-curve}. Then
\begin{align}
\|\eta-\tilde{\eta}\|_{L^\infty}&< 6\theta^{\frac18}r\quad\AND\quad
\label{eq:linfty-dist-theta}\\
\max\Big\{\sup_{y,z\in [x_i-r,x_+^i]\atop y\not= z}{\textstyle\frac{d_{\tilde{\eta}}(\tilde{\eta}(y),\tilde{\eta}(z))}{|\tilde{\eta}(y)-
\tilde{\eta}(z)|},}\sup_{y,z\in [x_-^i,x_i+r]\atop y\not= z}
{\textstyle\frac{d_{\tilde{\eta}}(\tilde{\eta}(y),\tilde{\eta}(z))}{|\tilde{\eta}(y)-
\tilde{\eta}(z)|}}\Big\}&<1+4\theta^\frac14\Foa i=1,\ldots,N.\notag
\end{align}
\end{lemma}
\begin{proof}
Note that $0<\theta_2=256^{-4}<\theta_1<\frac18$
so that Lemmas \ref{lem:quant-embed}
and \ref{lem:difference-quotients} are applicable here. The first inequality
in \eqref{eq:difference-quotients} for $y:=x_+^i\in G_+^i\subset B_r(x_i)$ 
and $x_0:=
x_-^i\in G_-^i$ for an arbitrary fixed $i\in\{1,\ldots,N\}$ implies for
$\nu_i:=\nu_{B_r(x_i)}\in\S^{n-1}$ as in \eqref{eq:bmo3}, and
$\Delta_i:=\frac{\eta(x_+^i)-\eta(x_-^i)}{x_+^i-x_-^i}\in \overline{B_1(0)}
\subset\R^n$
\begin{equation}\label{eq:nuDelta}
|\nu_i-\Delta_i|^2=1+|\Delta_i|^2-2\langle\nu_i,\Delta_i\rangle_{\R^n}\le
2\big(1-\langle\nu_i,\Delta_i\rangle_{\R^n}\big)
\overset{\eqref{eq:difference-quotients}}{\le}
4\theta^\frac14.
\end{equation}
Therefore, by virtue of \eqref{eq:bmo2} in Lemma \ref{lem:quant-embed} for
$\nu:=\nu_i$, $x:=x_i$, and $z\in [x_-^i,x_+^i]\subset B_r(x_i)$,
\begin{eqnarray*}
|\eta(z)-\tilde{\eta}(z)|& 
= &\textstyle
|\eta(z)-\eta(x_-^i)+\tilde{\eta}(x_-^i)-\tilde{\eta}(z)|=\big|\int_{x_-^i}^z
\big(\eta'(u)-
\tilde{\eta}'(u)\big)du\big|\\
& = & 
\textstyle
 \big|\int_{x_-^i}^z\big(\eta'(u)-\Delta_i\big)du\big|\le
2r\mvint_{B_r(x_i)}\big(|\eta'(u)-\nu_i|+|\nu_i-\Delta_i|\big)du\\
& \le & 
\textstyle
 2r\big(\mvint_{B_r(x_i)}|\eta'(u)-\nu_i|^2 du\big)^\frac12+2r|\nu_i-\Delta_i|\\
 & \overset{\eqref{eq:bmo2},\eqref{eq:nuDelta}}{<} &
 \textstyle
 2r\big(\sqrt{32\theta}
 +\sqrt{4\theta^\frac14}\big)=2r\theta^\frac18\big(\sqrt{32}
 \theta^\frac38+2\big)<6\theta^\frac18 r,
\end{eqnarray*}
which proves the first inequality in \eqref{eq:linfty-dist-theta}, since
$z\in [x_-^i,x_+^i]$ was arbitrarily chosen as well as the index
$i\in\{1,\ldots,N\}$.

Before proving the second inequality in \eqref{eq:linfty-dist-theta} let us insert
an auxiliary result comparing the intrinsic distance on $\eta$ with that on 
$\tilde{\eta}$, which turns out to be useful later, too. As in some
proofs before, we
abbreviate the intrinsic distance $d_\eta(\eta(y),\eta(z))$ on $\eta$ as 
$d_\eta(y,z)$, analogously on the modified curve $\tilde{\eta}$.
\begin{lemma}[Comparison of intrinsic distances]
\label{lem:compare-intrinsic}
Under the assumptions of Lemma \ref{lem:harmonic-substitution} one has
\begin{equation}\label{eq:compare-intrinsic}
\frac78 d_\eta(y,z)<(1-2\theta^\frac18)d_\eta(y,z)\le d_{\tilde{\eta}}(y,z)\le d_\eta(y,z)\quad\Foa
y,z\in\R/\Z, y\ne z,
\end{equation}
and the length $\mathscr{L}(\tilde{\eta})$ of $\tilde{\eta}$ is contained in the
interval $[1-2\theta^\frac18 , 1]\subset [\frac78,1]$.
\end{lemma}
\begin{proof}
As in the previous proof we set $\Delta_i:=\frac{\eta(x_+^i)-\eta(x_-^i)}{x_+^i-x_-^i}
\in\overline{B_1(0)}\subset\R^n$ and $\nu_i:=\nu_{B_r(x_i)}\in\S^{n-1}$ 
as in \eqref{eq:bmo3},
for $i=1,\ldots,N$. Use \eqref{eq:nuDelta} to infer
$$
\big|1-|\Delta_i|\big|=\big||\nu_i|-|\Delta_i|\big|\le |\nu_i-\Delta_i|\overset{
\eqref{eq:nuDelta}}{\le}2\theta^\frac18,
$$
so that the tangent vector of the modified curve $\tilde{\eta}$ satisfies
\begin{equation}\label{eq:modified-speed}
|\tilde{\eta}'(z)|=
\begin{cases}
|\eta'(z)| =1 & \textnormal{f.a.e.\@ $z\in\R/\Z\setminus\bigcup_{i=1}^N(x_-^i,x_+^i),$}\\
|\Delta_i|\in [1-2\theta^\frac18,1] & \textnormal{for all $z\in (x_-^i,x_+^i),i=1,\ldots,
N.$}
\end{cases}
\end{equation}
Consequently, we can estimate the length of $\tilde{\eta}$ as
\begin{equation}\label{eq:modified-length}
\mathscr{L}(\tilde{\eta})=\textstyle
\int_0 ^1|\tilde{\eta}'(u)|du\in [1-2\theta^\frac18,1]\subset [\frac78,1].
\end{equation}
For the proof of \eqref{eq:compare-intrinsic} we may assume $0\le y<z<1$.
If $d_\eta(y,z)=\int_y^z|\eta'(u)|du$ then by  \eqref{eq:modified-speed},
$$
d_{\tilde{\eta}}(y,z)\le\textstyle
\int_y^z|\tilde{\eta}'(u)|du\overset{\eqref{eq:modified-speed}}{\le} d_\eta(y,z).
$$
If, on the other hand, $d_\eta(y,z)=1-\int_y^z|\eta'(u)|du=\int_z^{1+y}|\eta'(u)|du$,
then 
$$
\textstyle
d_{\tilde{\eta}}(y,z)\le\int_z^{1+y}|\tilde{\eta}'(u)|du
\overset{\eqref{eq:modified-speed}}{\le}d_\eta(y,z),
$$ 
which proves the right
inequality in \eqref{eq:compare-intrinsic}.

Similarly, if $d_{\tilde{\eta}}(y,z)=\int_y^z|\tilde{\eta}'(u)|du$, then
$$
\textstyle
d_{\tilde{\eta}}(y,z)\overset{\eqref{eq:modified-speed}}{\ge}(1-2\theta^\frac18)
|z-y|_{\R/\Z}\ge
(1-2\theta^\frac18)
d_\eta(y,z)>\frac78 d_\eta(y,z).
$$
If $d_{\tilde{\eta}}(y,z)=\mathscr{L}(\tilde{\eta})-\int_y^z|\tilde{\eta}'(u)|du=
\int_z^{1+y}|\tilde{\eta}'(u)|du$, then
\[
\textstyle
d_{\tilde{\eta}}(y,z)\overset{\eqref{eq:modified-speed}}{\ge}(1-2\theta^\frac18)
\int_z^{1+y}|\eta'(u)|du\ge (1-2\theta^\frac18)d_\eta(y,z)>
\frac78 d_\eta(y,z).\qedhere
\]
\end{proof}
Now we return to the

\noindent
{\it Proof of Lemma \ref{lem:harmonic-substitution}.}\,
It remains to establish the estimate for the distortion quotient in
\eqref{eq:linfty-dist-theta}. Fix $i\in\{1,\ldots,N\}$, abbreviate
$x:=x_i$, $x_\pm:=x_\pm^i$, and consider the intervals
$\mathcal{A}:=[x-r,x_-]$, $\mathcal{B}:=[x_-,x_+]$, and $\mathcal{C}:=[x_+,x+r]$.
If $y,z\in \mathcal{B}$ 
then $d_{\tilde{\eta}}(y,z)=|\tilde{\eta}(y)-\tilde{\eta}(z)|$ due to $r<\frac14$,
which implies the right inequality in \eqref{eq:linfty-dist-theta},
since
we are on the straight segment in that case. 

So, we distinguish two remaining cases:
\begin{enumerate}
\item[\rm I.] $(y,z)\in \mathcal{A}\times \mathcal{A},$ (or $(y,z)\in 
\mathcal{C}\times \mathcal{C}$),
\item[\rm II.]
$(y,z)\in \mathcal{A}\times \mathcal{B},$ (or $(y,z)\in 
\mathcal{B}\times \mathcal{C}$, or  $\mathcal{C}
\times \mathcal{B}$, or $\mathcal{B}\times \mathcal{A}$).
\end{enumerate}
It suffices in each case to treat the situation named first, the alternatives in brackets
can then be proved analogously.

\noindent
{\it Case I.}\,
For any $\zeta\in \mathcal{A}$ one has
\begin{align*}
|\zeta-x|&\textstyle
\ge |x_- - x|=|x_- -(x-\frac{r}2)+(x-\frac{r}2)-x|\ge\frac{r}2-|x_- -(x-\frac{r}2)|
 >\frac{r}2 -\frac{r}8=\frac38 r>\theta r,
\end{align*}
so that $\mathcal{A}
\subset A_{r,\theta}(x)=B_r(x)\setminus B_{\theta r}(x)$. Therefore, our
assumption \eqref{eq:fractional-seminorm-theta3} implies that we
can apply
Corollary 
\ref{cor:seminorm-bilip} to the interval  $I:=\mathcal{A}$, $\g:=\eta|_\mathcal{A}$,
to deduce from \eqref{eq:seminorm-bilip2} 
by definition \eqref{eq:modified-curve} of $\tilde{\eta}$
\begin{equation}\label{eq:dist-A}
\textstyle
\frac{d_{\tilde{\eta}}(y,z)}{|\tilde{\eta}(y)-\tilde{\eta}(z)|}
\overset{\eqref{eq:modified-curve}}{=}
\frac{d_{\eta}(y,z)}{|\eta(y)-\eta(z)|}
=\frac{|y-z|_{\R/\Z}}{|\eta(y)-\eta(z)|}
\overset{\eqref{eq:seminorm-bilip2}}{<}
\frac1{\sqrt{1-\frac{\theta}2}}<1+\theta^\frac14.
\end{equation}

\noindent
{\it Case II.}\,
For $(y,z)\in \mathcal{A}\times \mathcal{B}$ 
and $\nu:=\nu_{B_r(x)}\in\S^{n-1}$ as defined 
in \eqref{eq:bmo3} we infer from the definition \eqref{eq:modified-curve},
\eqref{eq:difference-quotients} of
Lemma \ref{lem:difference-quotients}, 
and \eqref{eq:compare-intrinsic} of Lemma
\ref{lem:compare-intrinsic}
\begin{align*}
\textstyle
|\tilde{\eta}(z)-\tilde{\eta}(y)|& 
\textstyle
\ge
\langle\tilde{\eta}(z)-\tilde{\eta}(y),\nu\rangle_{\R^n}=
\langle\tilde{\eta}(x_-)-\tilde{\eta}(y),\nu\rangle_{\R^n}+
\langle\tilde{\eta}(z)-\tilde{\eta}(x_-),\nu\rangle_{\R^n}\\
&\overset{\eqref{eq:modified-curve}}{=}\textstyle
\langle\tilde{\eta}(x_-)-\tilde{\eta}(y),\nu\rangle_{\R^n}+
\frac{z-x_-}{x_+-x_-}\langle\eta(x_+)-\eta(x_-),\nu\rangle_{\R^n}\\
&\overset{\eqref{eq:difference-quotients}}{\ge}
\textstyle
(1-2\theta^\frac14)\big[(x_--y)+(z-x_-)\big]=
(1-2\theta^\frac14)(z-y)\\
&\textstyle =(1-2\theta^\frac14)d_\eta(y,z)\overset{\eqref{eq:compare-intrinsic}}{\ge}(1-2\theta^\frac14)d_{\tilde{\eta}}(y,z),
\end{align*}
which implies
$$
\textstyle
\frac{d_{\tilde{\eta}}(y,z)}{|\tilde{\eta}(y)-\tilde{\eta}(z)|}
\le\frac1{1-2\theta^\frac14}=1+\frac{2\theta^\frac14}{1-2\theta^\frac14}<
1+4\theta^\frac14,
$$
because $0<\theta<\theta_2=256^{-4}.$
\end{proof}
Finally we show how to control the distortion quotient globally
under harmonic substitution, which may be reinterpreted as
controlling the bilipschitz constant of $\tilde{\eta}$ measured
in its intrinsic distance $d_{\tilde{\eta}}(\cdot,\cdot)$.
\begin{lemma}[Bilipschitz control under harmonic substitution]
\label{lem:bilip-control-harmonic}
Let $\eta\in W^{\frac32,2}(\R/\Z,\R^n)$ be an arclength parametrized
curve satisfying \eqref{eq:fractional-seminorm-theta3} for finitely
many distinct parameters $x_1,\ldots,x_N\in\R/\Z$ with
$r\in (0,\frac14)$ as in \eqref{eq:little-r} and $\theta\in (0,\theta_3]$,
where
\begin{equation}\label{eq:theta3}
\theta_3:=\frac1{4^{24}L^8}=\big(\frac1{64L}\big)^8,
\end{equation}
with $L\in [1,\infty)$ such that
\begin{equation}\label{eq:bilip-eta}
|\eta(y)-\eta(z)|\ge\frac1{L}|y-z|_{\R/\Z}\quad\Foa y,z\in\R/\Z.
\end{equation}
If $\tilde{\eta}$ is the modified curve obtained from $\eta$ by
harmonic substitution as in \eqref{eq:modified-curve}, then
there exists some number $\tilde{L}\in [L,2L]$ such that
\begin{equation}\label{eq:bilip-tilde-eta}
|\tilde{\eta}(y)-\tilde{\eta}(z)|\ge\frac1{\tilde{L}}d_{\tilde{\eta}}(y,z)
\quad\Foa y,z\in\R/\Z.
\end{equation}
\end{lemma}
\begin{proof}
Note that the smallness assumption \eqref{eq:theta3} implies
$0<\theta\le\theta_3<\theta_2=256^{-4}<\theta_1=144^{-4},$ so that
all previous lemmas in the present section are applicable. We
distinguish three (partly overlapping) cases:
\begin{enumerate}
\item[\rm I.] $y,z\in [x_i-r,x_+^i]$ (or $y,z\in [x^i_-,x_i+r]$) for some 
$i\in\{1,\ldots,N\}$,
\item[\rm II.] $y,z\in\R/\Z\setminus \bigcup_{i=1}^N[x_-^i,x_+^i],$
\item[\rm III.] $y\in [x_-^i,x_+^i]$, $z\in\R/\Z\setminus
[x_i-r,x_i+r]$ for some $i\in\{1,\ldots,N\}$.
\end{enumerate}
\noindent
{\it Case I.}\,
We apply \eqref{eq:linfty-dist-theta} of Lemma 
\ref{lem:harmonic-substitution} to deduce for $d_{\tilde{\eta}}(y,z)\equiv
d_{\tilde{\eta}}(\tilde{\eta}(y),\tilde{\eta}(z))$
\begin{equation}\label{eq:bilip-tilde1}
\textstyle
\frac{1024}{1025}d_{\tilde{\eta}}(y,z)\le
\frac{d_{\tilde{\eta}}(y,z)}{1+4\theta^\frac14}\overset{\eqref{eq:linfty-dist-theta}}{<}|\tilde{\eta}(y)-\tilde{\eta}(z)|\quad\Foa y,z\in [x_i-r,x_+^i]
(\textnormal{or $[x_-^i,x_i+r]$}).
\end{equation}
\noindent
{\it Case II.}\,
For $y,z\in\R/\Z\setminus \bigcup_{i=1}^N[x_-^i,x_+^i]$ we have by 
means of \eqref{eq:compare-intrinsic} $d_{\tilde{\eta}}(y,z)
\le d_\eta(y,z)$ and $|\tilde{\eta}(y)-\tilde{\eta}(z)|=|\eta(y)-\eta(z)|$,
so that by assumption \eqref{eq:bilip-eta} 
\begin{equation}\label{eq:bilip-tilde2}
\textstyle
\frac1{L}d_{\tilde{\eta}}(y,z)\le
\frac1{L}d_{\eta}(y,z)=\frac1{L}|y-z|_{\R/\Z}
\overset{\eqref{eq:bilip-eta}}{\le}
|\eta(y)-\eta(z)|=
|\tilde{\eta}(y)-\tilde{\eta}(z)|.
\end{equation}
\noindent
{\it Case III.}\,
For $y\in [x_-^i,x_+^i]$ and $z\notin [x_i-r,x_i+r]$ we find by
\eqref{eq:compare-intrinsic}, the first inequality in
\eqref{eq:linfty-dist-theta} of Lemma \ref{lem:harmonic-substitution},
and
assumptions \eqref{eq:bilip-eta} and \eqref{eq:theta3}
\begin{align}
\textstyle\frac{d_{\tilde{\eta}}(y,z)}{|\tilde{\eta}(y)-\tilde{\eta}(z)|}&
\textstyle
\stackrel{\eqref{eq:linfty-dist-theta}, \eqref{eq:compare-intrinsic}}{<}
\frac{d_{\eta}(y,z)}{|\eta(y)-\eta(z)|-12\theta^\frac18 r}=\big(
\frac{|\eta(y)-\eta(z)|}{|y-z|_{\R/\Z}}-\frac{12\theta^\frac18
r}{
|y-z|_{\R/\Z}}\big)^{-1}\overset{\eqref{eq:bilip-eta}}{<}
\big(\frac1{L}-\frac{12\theta^\frac18}{3/8}\big)^{-1}
\overset{\eqref{eq:theta3}}{\le}2L,
\end{align}
since $|y-z|_{\R/\Z}\ge |z-x|_{\R/\Z}-\max\{|x-x_-^i|,
|x-x_+^i|\}
>r-\frac58 r=\frac38 r.$
\end{proof}

\section{Weak fractional compactness}\label{sec:compactness}
From the introduction we recall  
Definition~\eqref{eq:symmetric-subset} of
the $p$-rotationally symmetric subset $\Sigma_p(\mathcal{K})$ of a given
knot equivalence class $\mathcal{K}$ for $p\in\N\setminus\{1\}$. 
Consider now for given constants
$L\ge 1$ and $R>0$ the subset
\begin{equation}\label{eq:compact-subset}
\mathcal{F}_{p,L,R}(\mathcal{K}):=
\{\eta\in\Sigma_p(\mathcal{K})\!\cap\! W^{\frac32,2}(\R/\Z,\R^3):
|\eta'|\!=\!1\,\textnormal{a.e.},\,\bilip(\eta)\!\ge\!
\textstyle\frac1{L},\,\max\{|\eta(0)|,\lfloor\eta'
\rfloor^2_{\frac12,2}\}\!\le\! R\},
\end{equation}
where $\bilip(\eta)$ denotes the bilipschitz constant defined in
\eqref{eq:def-bilip} in Section~\ref{sec:stability}. Then
we can prove the following compactness result which may be of
independent interest in geometric knot theory.
\begin{theorem}[Weak fractional compactness]
\label{thm:fractional-compactness}
For any tame prime knot equivalence class  $\mathcal{K}$ the set
$\mathcal{F}_{p,L,R}(\mathcal{K})$ is weakly sequentially compact
in $W^{\frac32,2}(\R/\Z,\R^3)$ in the restricted sense that for any sequence
$(\g_n)_n\subset \mathcal{F}_{p,L,R}(\mathcal{K})\cap C^1(\R/\Z,\R^3)$ 
there is a knot
$\g\in\mathcal{F}_{p,L,R}(\mathcal{K})$ and a 
subsequence $\g_{n_k}\rightharpoonup \g$ in $W^{\frac32,2}$ as $k\to\infty$.
\end{theorem}

Strictly speaking, the tameness requirement is superfluous since
$\mathcal{F}_{p,L,R}(\mathcal{K})=\emptyset$ if $\mathcal K$ is wild
as pointed out at the end of the proof of Corollary \ref{cor:fractional-stability}.
The additional smoothness assumption on the sequence $(\g_n)_n$ guarantees 
that sequence members remain tame after the harmonic substitution described
in Section~\ref{sec:harmonic}. This assumption might 
be of purely technical
nature and can possibly be omitted.

For our application 
in search of symmetric minimizing knots for the M\"obius energy
in Section~\ref{sec:symm-critical}, however, a mollification argument
allows us to focus on smooth minimizing sequences satisfying all assumptions
of Theorem~\ref{thm:fractional-compactness} above.

\noindent
{\it Proof of Theorem~\ref{thm:fractional-compactness}.}\,
We proceed in several steps.

{\it Step 1. Subconvergence to a candidate curve $\g$.}\,
Since $|\g_n'|=1$ a.e.\@ on $\R/\Z$ and $|\g_n(t)|\le
|\g_n(t)-\g_n(0)|+R\le 1+R$ for a.e.\@ $t\in\R/\Z$ one finds a constant
$C=C(R)$ independent of $n$ such that
\begin{equation}\label{eq:w3halbe-bound}
\|\g_n\|_{W^{\frac32,2}}\le C\quad\Foa n\in\N.
\end{equation}
The space $W^{\frac32,2}(\R/\Z,\R^3)$ is reflexive so that we may
assume that there is some curve $\g\in W^{\frac32,2}
(\R/\Z,\R^3)$ such that (up to a subsequence)
\begin{equation}\label{eq:weak-subconvergence}
\g_n\rightharpoonup \g \quad\textnormal{as $n\to\infty$.}
\end{equation}
By the Poincar\'e inequality (cf.~\cite[Proposition A.3]{knappmann-etal_2022}) the seminorm $\lfloor\cdot
\rfloor_{\frac12,2}$ is a norm on the subspace $\{f\in 
W^{\frac12,2}(\R/\Z,\R^3):\int_{\R/\Z}f(u)du=0\}$ 
equivalent to $\|\cdot\|_{W^{\frac12,2}}$. Since $
\int_{\R/\Z}\eta'(u)du=0$
for all $\eta\in W^{\frac32,2}(\R/\Z,\R^3)$ we therefore have by 
\eqref{eq:weak-subconvergence},
\begin{equation}\label{eq:uhs-seminorm}
\lfloor\g'\rfloor_{\frac12,2}\le\liminf_{n\to\infty}\lfloor\g_n'
\rfloor_{\frac12,2}\le \sqrt{R}.
\end{equation}
Moreover, due to the compact embeddings $W^{\frac32,2}\hookrightarrow
C^0$ and $W^{\frac32,2}\hookrightarrow
W^{1,1}$ and the theorem of Fischer--Riesz, we find (up to a further
subsequence)
\begin{equation}\label{eq:unif-subconvergence}
\g_n\to\g\quad\textnormal{in $C^0$, \quad and}\quad
\g_n'\to\g'\quad\textnormal{a.e.\@ on $\R/\Z$\quad as $n\to\infty$.}
\end{equation}
This implies $|\g(0)|=\lim_{n\to\infty}|\g_n(0)|\le R$ and
$|\g'|=\lim_{n\to\infty}|\g_n'|=1$ a.e.\@ on $\R/\Z$. The uniform
convergence in \eqref{eq:unif-subconvergence} leads to the bilipschitz
estimate
$$
\textstyle
|\g(u)-\g(v)|=\lim_{n\to\infty}
|\g_n(u)-\g_n(v)|\ge\frac1{L}|u-v|_{\R/\Z}\quad\Foa u,v\in\R/\Z;
$$
hence 
\begin{equation}\label{eq:bilip-gamma}
\textstyle
\bilip(\g)\ge\frac1{L},
\end{equation}
and $\g$ is embedded. 
Also the symmetry relation \eqref{eq:rot-symm} is preserved in the uniform
limit
since
$$
\textstyle
\g(s)=\lim_{n\to\infty}\g_n(s)\overset{\eqref{eq:rot-symm}}{=}
R_{\frac{2\pi}p}\lim_{n\to\infty}\g_n(s-\frac1{p})=
R_{\frac{2\pi}p}\g(s-\frac1{p})\quad\Foa s\in\R/\Z.
$$
So, the only requirement on the limiting knot $\g$ to be checked, 
is that $\g$ represents the prescribed knot equivalence
class, i.e., that $[\g]=\mathcal{K}$. 

{\it Step 2. Controlling the local distortion of the limiting  knot $\g$.}\,
Combining \eqref{eq:seminorm-bilip2} in Corollary \ref{cor:seminorm-bilip} 
with the absolute
continuity of the double integral defining the seminorm
 $\lfloor\g'\rfloor_{\frac12,2}\le$ $\sqrt{R}$ yields a number $\rho_\g\in
 (0,\frac14)$ such that for $d_\g(s,t)\equiv d_\g(\g(s),\g(t))=|s-t|_{\R/\Z}$
 \begin{equation}\label{eq:dist-limit}
\textstyle
\sup\big\{\frac{d_\g(s,t)}{|\g(s)-\g(t)|}:s,t\in\R/\Z,\,0<|s-t|_{\R/\Z}=
|s-t|\le 2\rho_\g\big\}<g_3.
\end{equation}
For parameters $s,t\in\R/\Z$ with $d_\g(s,t)>2\rho_\g$, on the
other hand, we infer from \eqref{eq:bilip-gamma}
\[ 2\rho_\g<d_\g(s,t)=|s-t|_{\R/\Z}\le L|\g(s)-\g(t)|. \]
This together
with \eqref{eq:dist-limit} leads to 
\begin{equation}\label{eq:local-distortion-limit}
\textstyle
\delta(\g,r_\g)<g_3\quad\textnormal{at scale
$r_\g:=\frac{\rho_\g}L.$}
\end{equation}
Suppose we had $\delta(\g_m,r_{\g_m})<g_3$ for  some $m\in\N$ at some
scale $r_{\g_m}>0$, then we could conclude the knot equivalence $
\mathcal{K}=[\g_m]=[\g]$ from Corollary \ref{cor:equiv-knots},
if, in addition, $\|\g_m-\g\|_{L^\infty}<\frac14\min\{r_\g,r_{\g_m}\}$.
But as indicated in Section~\ref{sec:strategy} of the introduction
this might be too much to hope for because of possible 
concentrations in the seminorms $\lfloor \g_m'\rfloor_{\frac12,2}$
as $m\to\infty$. Therefore, we are going to carefully
cut out possible
concentrations of the curves $\g_n$,
 replace these subarcs by straight segments
to obtain modified curves $\tilde{\g}_n$, and use
the estimates of Section~\ref{sec:harmonic} to finally
find a sequence member $\tilde{\g}_m$ sufficiently 
$L^\infty$-close to $\g$, so that Corollary \ref{cor:equiv-knots}
indeed implies
$[\g]=[\tilde{\g}_m]=[\g_m]=\mathcal{K}$. 
The details are presented in the remaining steps. This is 
where the assumptions $\g_n\in C^1(\R/\Z,\R^3)$ and $\mathcal{K} $
prime come into play.

{\it Step 3. Localizing concentrations via weak convergence of measures.}\,
Fix
\begin{equation}\label{eq:fix-theta}
\textstyle
\theta\equiv \theta_4:=\big(\frac1{6\cdot 128L}\big)^8<\theta_3<\theta_2<
\theta_1<\frac1{20},
\end{equation}
such that we have Lemmas \ref{lem:quant-embed}, 
\ref{lem:difference-quotients}, \ref{lem:harmonic-substitution},
\ref{lem:compare-intrinsic}, and \ref{lem:bilip-control-harmonic}
of Section~\ref{sec:harmonic} at our disposal. In addition, set
\begin{equation}\label{eq:epsilon-choice}
\varepsilon:=\frac23-\frac6{\pi^2}>0
\end{equation}
and define the Radon measures\footnote{suitably extended to
an outer measure; see, e.g., \cite[pp. 13,14]{teschl_2025}.}
\begin{equation}\label{eq:def-radon}
\textstyle
\mu_n(A):=\iint_A\frac{|\g_n'(u)-\g_n'(v)|^2}{|u-v|^2}\,dudv
\quad
\text{for Lebesgue measurable}
A\subset\R/\Z\times\R/\Z,\,\,n\in\N.
\end{equation}
By means of the assumed uniform bound $\lfloor\g_n'\rfloor_{\frac12,2}\le
\sqrt R$ 
we find $\sup_{n\in\N}\mu_n(\R/\Z\times\R/\Z)\le R <\infty$,
so that there is a Radon measure $\mu$ on $\R/\Z\times\R/\Z$ such that
(up to a subsequence)
\begin{equation}\label{eq:weak-convergence}
\mu_n\rightharpoonup\mu\quad\As n\to\infty,
\end{equation}
cf.\@ \cite[Theorem~1.41]{evans-gariepy_2015}. (Note that the limiting
measure $\mu$ might be different from $\lfloor\g'\rfloor_{\frac12,2,\cdot}^{2}$.)
Since the mapping $r\mapsto\mu(B_r(x)\times B_r(y))$ is non-negative and 
non-decreasing for fixed parameters $x,y\in\R/\Z$, its limit
\begin{equation}\label{eq:limit-balls}
\textnormal{$\lim_{r\downarrow 0}\mu(B_r(x)\times B_r(y))$ exists and is 
non-negative,}
\end{equation}
so that we can consider the set
\begin{equation}\label{eq:W-eps}
W_\varepsilon:=\big\{ (x,y)\in\R/\Z\times\R/\Z:\lim_{r\downarrow 0}\mu\big(B_r(x)
\times B_r(y)\big)>\varepsilon\big\}.
\end{equation}
We claim that concentrations take place only on the diagonal, i.e.,
\begin{equation}\label{eq:diagonal}
W_\varepsilon=\bigcup_{x\in V_\varepsilon} \{x\}\times\{x\},
\end{equation}
where $V_\varepsilon:=\{x\in\R/\Z:\lim_{r\downarrow 0}\mu(B_r(x)\times B_r(x))>
\varepsilon\}$.
Indeed, for arbitrary fixed $n\in\N$, distinct parameters
$x,y\in\R/\Z$, and $0<r<\frac14 |x-y|_{\R/\Z},$ one has
$|u-v|\ge |x-y|-2r >\frac12 |x-y|_{\R/\Z}$ for all $u\in B_r(x)$,
$v\in B_r(y)$, so that (with $|\g_n'|=1$ a.e.)
\begin{equation}\label{eq:diagonal-proof}
\mu_n\big(B_r(x)\times B_r(y)\big) < \frac4{|x-y|_{\R/\Z}^2}\cdot 4(2r)^2,
\end{equation}
and therefore, by the characterization of weak convergence
of measures \cite[Theorem~1.40]{evans-gariepy_2015},
$$
\textstyle
\mu\big(B_r(x)\times B_r(y)\big)\le\liminf_{n\to\infty}
\mu_n\big(B_r(x)\times B_r(y)\big)\overset{\eqref{eq:diagonal-proof}}{\le}
\frac{64}{|x-y|^2_{\R/\Z}}\cdot r^2\to 0\quad\As r\to 0.
$$
Hence $(x,y)\not\in W_\varepsilon$, which proves our claim 
\eqref{eq:diagonal}.

We also claim that $W_\varepsilon$ contains only finitely many elements,
that is,
\begin{equation}\label{eq:finitely-many}
\textstyle
\sharp W_\varepsilon =\sharp V_\varepsilon\le\lceil \frac{R}\varepsilon
\rceil,
\end{equation}
where $\lceil a\rceil$ denotes the smallest integer greater or equal to $a$.
For that take $j$ parameters $x_1,\ldots, x_j\in V_\varepsilon$ and
choose radii $r_1,\ldots,r_j>0$ such that
$\mu(B_{r_k}(x_k)\times B_{r_k}(x_k))>\varepsilon$ and
$B_{r_i}(x_i)\cap B_{r_l}(x_l)=\emptyset $ for all $i,k,l\in
\{1,\ldots,j\}$ with $i\not= l$.
Then
$$
\textstyle
\varepsilon\cdot j< \sum_{k=1}^j
\mu\big(B_{r_k}(x_k)\times B_{r_k}(x_k)\big)\le
\mu(\R/\Z\times \R/\Z)\le \liminf_{n\to\infty}\mu_n\big(\R/\Z\times
\R/\Z\big)\le R,
$$
which establishes \eqref{eq:finitely-many}.

Our last claim in this step says that the measure $\mu$ is uniformly
small on small sets off $W_\varepsilon$. More precisely,
\begin{equation}\label{eq:unif-small-mu}
\textnormal{There exists $\rho_\mu>0$, such that
$\mu\big((B_{\rho_\mu}(z)\times B_{\rho_\mu}(z))\setminus W_\varepsilon\big)
\le 2\varepsilon$\,\, for all $z\in\R/\Z.$}
\end{equation}
To prove this define the restricted measures 
(cf.\@ \cite[Definition 1.2]{evans-gariepy_2015}) 
$\bar{\mu}:=\mu\measurerestr [(\R/\Z\times\R/\Z)\setminus
W_\varepsilon ]$ satisfying
\begin{equation}\label{eq:restricted-measure}
\bar{\mu}(A)=\mu\big([(\R/\Z\times\R/\Z)\setminus W_\varepsilon ]\cap A\big)
=\mu\big(A\setminus W_\varepsilon)\quad
\text{for Lebesgue measurable }
A\subset\R/\Z\times\R/\Z,
\end{equation}
and assume for contradiction, that for every $k\in\N$ there is
a parameter $z_k\in\R/\Z$ such that $\bar{\mu}\big(B_\frac1k (z_k)
\times B_\frac1k (z_k)\big)>2\varepsilon.$ We may assume that $z_k\to
z_0\in\R/\Z$ as $k\to\infty$, so that for every $r>0$ there is an 
index $k_0\in\N$ such that 
$|z_k-z_0|<\frac{r}2$ for all $k\ge k_0$, which implies for all 
$k\ge\max\{k_0,\frac2{r}\}$
$$
|u-z_0|\le |u-z_k|+|z_k-z_0|<\frac1{k}+\frac{r}2\le r\quad\Foa
u\in B_\frac1k (z_k).
$$
Consequently,
\begin{equation}\label{eq:widerspruch-mubar}
\textstyle
\bar{\mu}\big( B_r(z_0)\times B_r(z_0)   \big)
\ge\bar{\mu}\big( B_\frac1k (z_k)\times B_\frac1k (z_k)  \big)
>2\varepsilon\quad\Foa r>0,\,k\ge\max\{k_0,\frac2{r}\}.
\end{equation}
{\it Case 1.}\,
If $z_0\not\in V_\varepsilon$ then by definition of $W_\varepsilon,$
by \eqref{eq:diagonal}, \eqref{eq:restricted-measure}, and
\eqref{eq:widerspruch-mubar},
$$
\varepsilon\ge\lim_{r\downarrow 0}\mu\big(B_r(z_0)\times B_r(z_0) \big)
\overset{\eqref{eq:restricted-measure}}{\ge}
\lim_{r\downarrow 0}\bar{\mu}\big(B_r(z_0)\times B_r(z_0)\big)
\overset{\eqref{eq:widerspruch-mubar}}{\ge}2\varepsilon,
$$
contradiction.

{\it Case 2.}\,
If $z_0\in V_\varepsilon$, then write $\{z_0\}\times\{z_0\}=
\bigcap_{r>0}(B_r(z_0)\times B_r(z_0))$ to obtain by virtue
of \eqref{eq:diagonal}, \eqref{eq:restricted-measure}, 
\cite[Theorem~1.2(iv)]{evans-gariepy_2015}, and \eqref{eq:widerspruch-mubar},
\begin{align*}
0 & =
\mu(\emptyset)\overset{\eqref{eq:diagonal}}{=}
\mu\big((\{z_0\}\times\{z_0\})\setminus W_\varepsilon\big)
\overset{\eqref{eq:restricted-measure}}{=}\textstyle
\bar{\mu}\big(\bigcap_{r>0}(B_r(z_0)\times B_r(z_0)\big)\\
&
 = \lim_{r\downarrow 0}\bar{\mu}\big(B_r(z_0)\times B_r(z_0)\big)
\overset{\eqref{eq:widerspruch-mubar}}{\ge} 2\varepsilon,
\end{align*}
again a contradiction. This completes the proof of our claim
\eqref{eq:unif-small-mu}.

{\it Step 4. Identifying a suitable approximating knot $\g_{n_0}$.}\,
For any fixed parameter $x\in\R/\Z$ we estimate the measure $\mu$ on annuli
$A_{r,\theta}(x)=B_r(x)\setminus B_{\theta r}(x)$ as
\begin{align*}
0 & \le \mu\big( \overline{A_{r,\theta}(x)}\times
\overline{A_{r,\theta}(x)}\big)\le
\mu\big((B_{2r}(x)\setminus B_{\theta r}(x))\times
(B_{2r}(x)\setminus B_{\theta r}(x))\big)\\
& = \mu\big( B_{2r}(x)\times B_{2r}(x) \big)-
\mu\big( B_{\theta r}(x)\times B_{\theta r}(x)\big)
\to 0\quad\As r\to 0
\end{align*}
by virtue of \eqref{eq:limit-balls}. Thus, there is a radius $r_x>0$
such that
\begin{equation}\label{eq:small-annuli-measure}
\mu\big( \overline{A_{r,\theta}(x)}\times
\overline{A_{r,\theta}(x)} \big) <\textstyle\frac{\theta}2\quad\Foa
r\in (0,r_x],
\end{equation}
where $\theta$ is given in \eqref{eq:fix-theta}.
In particular, for the finitely many parameters in
$V_\varepsilon=\{x_1,\ldots,x_N\}$, $N\le\lceil\frac{R}\varepsilon
\rceil$ we find radii $r_l:=r_{x_l}>0$ such that
\begin{equation}\label{eq:small-annuli-measure-xl}
\mu\big( \overline{A_{r,\theta}(x_l)}\times
\overline{A_{r,\theta}(x_l)} \big) 
<\textstyle\frac{\theta}2\quad\Foa
r\in (0,r_l],\,\,l=1,\ldots,N.
\end{equation}
Now set
\begin{equation}\label{eq:final-radius}
\bar{r}:=
\textstyle
\min\big\{\rho_\mu,r_1,\ldots,r_N,r_\g,\frac1{4p},\frac14\min\{|x_i-x_k|:
i\not=k,i,k\in\{1,\ldots,N\}\,\big\} >0,
\end{equation}
where $\rho_\mu>0$ is the radius such that \eqref{eq:unif-small-mu}
holds, and $r_\g>0$ is the scale at which $\g$ satisfies the
distortion estimate \eqref{eq:local-distortion-limit}, and $p\in\N\setminus
\{1\}$ determines the prescribed rotational symmetry defined in \eqref{eq:rot-symm}.
Inequality \eqref{eq:small-annuli-measure-xl} implies in particular that
\begin{equation}\label{eq:small-annuli-measure-xl-unif}
\mu\big( \overline{A_{r,\theta}(x_l)}\times
\overline{A_{r,\theta}(x_l)} \big)
<\textstyle\frac{\theta}2\quad\Foa
r\in (0,\bar{r}],\,\,l=1,\ldots,N.
\end{equation}
For the fixed radius $r:
=\bar{r}$ we use the weak convergence \eqref{eq:weak-convergence} to find for each $l\in\{1,\ldots,N\}$ some index $n_l\in\N$ such
that by \cite[Theorem~1.40]{evans-gariepy_2015}
$\mu_n(\overline{A_{\bar{r},\theta}(x_l)}\times
\overline{A_{\bar{r},\theta}(x_l)})<\theta$ for all $n\ge n_l$, so that
\begin{equation}\label{eq:unif-small-measures-annuli-indices}
\mu_n\big(\overline{A_{\bar{r},\theta}(x_l)}\times
\overline{A_{\bar{r},\theta}(x_l)})
<\theta\quad\Foa l=1,\ldots,N,\,n\ge n_*,
\end{equation}
where $n_*:=\max\{n_l:l=1,\ldots,N\}$.
By means of the uniform convergence in \eqref{eq:unif-subconvergence}
we may choose $n_{**}\in\N,$ $n_{**}\ge n_*$ such that
\begin{equation}\label{eq:linfty-close-approx}
\|\g_n-\g\|_{L^\infty}\le\textstyle \frac1{128}\cdot\frac{\bar{r}}L
\quad\Foa n\ge n_{**}.
\end{equation}
Now choose finitely many parameters $y_1,\ldots,y_K\in \R/\Z\setminus
B_{\bar{r}}(V_\varepsilon)$, such that
\begin{equation}\label{eq:covering-complement-Veps}
\R/\Z\setminus \overline{B_{\frac34 \bar{r}}(V_\epsilon)}
\subset\bigcup_{k=1}^K B_{\frac14 \bar{r}}(y_k).
\end{equation}
In particular, from $y_1,\ldots,y_K\notin B_{\bar{r}}(V_\varepsilon)$ we infer that
\begin{equation}\label{eq:covering-intersect}
B_{\frac58 \bar{r}}(V_\epsilon)\cap\bigcup_{k=1}^K B_{\frac38 \bar{r}}(y_k)
=\emptyset.
\end{equation}
Then use the weak convergence \eqref{eq:weak-convergence} again to
find an index $n_0\ge n_{**}$ such that in view of \eqref{eq:unif-small-mu}
\begin{equation}\label{eq:max-mu-n0}
\max_{k=1,\ldots,K}\textstyle
\mu_{n_0}\big(\overline{B_{\frac12 \bar{r}}(y_k)}\times
\overline{B_{\frac12 \bar{r}}(y_k)}\big)\le 3\varepsilon,
\end{equation}
which is possible, since by \cite[Theorem~1.40]{evans-gariepy_2015} and
by our choice of $\bar{r}$ in \eqref{eq:final-radius} and of the
parameters $y_1,\ldots,y_K\in\R/\Z\setminus B_{\bar{r}}(V_\varepsilon)$,
\begin{align*}
&\limsup_{n\to\infty}\mu_n\big(\overline{B_{\frac12 \bar{r}}(y_k)}
\times
\overline{B_{\frac12 \bar{r}}(y_k)}\big)
\overset{\eqref{eq:weak-convergence}}{\le}
\mu\big(\overline{B_{\frac12 \bar{r}}(y_k)}
\times
\overline{B_{\frac12 \bar{r}}(y_k)}\big)\le
\mu\big(B_{\frac34 \bar{r}}(y_k)
\times B_{\frac34 \bar{r}}(y_k)\big)\\
& = \mu\big(B_{\frac34 \bar{r}}(y_k)
\times B_{\frac34 \bar{r}}(y_k)\setminus W_\varepsilon\big)
\overset{\eqref{eq:final-radius}}{\le}
\mu\big((B_{\rho_\mu}(y_k)
\times
B_{\rho_\mu}(y_k))\setminus W_\varepsilon\big)
\overset{\eqref{eq:unif-small-mu}}{\le}
2\varepsilon\Foa k=1,\ldots,K.
\end{align*}
Summarizing \eqref{eq:unif-small-measures-annuli-indices},
\eqref{eq:linfty-close-approx}, and \eqref{eq:max-mu-n0},
we arrive at
\begin{equation}\label{eq:summary-n0}
\textstyle
\begin{cases}
\mu_{n_0}\big(\overline{A_{\bar{r},\theta}(x_l)}\times
\overline{A_{\bar{r},\theta}(x_l)}\big)<\theta & \Foa l=1,\ldots,N,\\
\mu_{n_0}\big(B_{\frac{\bar{r}}2}(y_k)\times
B_{\frac{\bar{r}}2}(y_k)\big)\le 3\varepsilon &\Foa k=1,\ldots,K,\\
\|\g_{n_0}-\g\|_{L^\infty} \le \frac1{128}\cdot\frac{\bar{r}}L.
\end{cases}
\end{equation}

{\it Step 5. Cutting out concentrations of $\g_{n_0}$ and harmonic 
substitution.}\,
Now we modify the knot
$\g_{n_0}$ in each ball $B_{\bar{r}}(x_i)\subset\R/\Z$, $x_i\in
V_\varepsilon,$ $i=1,\ldots,N,$ where $\bar{r}$ is defined in 
\eqref{eq:final-radius}. We do that as in \eqref{eq:modified-curve}
by choosing suitable truncation points $x_\pm^i\in G_\pm^i:=
G_\pm(x_i,\theta,\bar{r})$, where the sets $G_\pm^i$ are defined
as in \eqref{eq:G-set} for $x:=x_i$, $r:=\bar{r}$, and
$\theta:=\theta_4<\theta_3=(64L)^{-8}$ (see  \eqref{eq:fix-theta}),
to obtain the modified curve
\begin{equation}\label{eq:gamma-n0-modified}
\tilde{\gamma}_{n_0}(z):=\begin{cases}
\gamma_{n_0}(x_-^i)+\frac{z-x_-^i}{x_+^i-x_-^i}\big(\gamma_{n_0}(x_+^i)
-\gamma_{n_0}(x_-^i)\big) & \Fo
z\in [x_-^i,x_+^i], i=1\ldots,N,\\
\gamma_{n_0}(z) & \Fo z\in\R/\Z\setminus\bigcup_{i=1}^N[x_-^i,x_+^i].
\end{cases}
\end{equation}
If $V_\varepsilon=\emptyset$ then simply set $\tilde{\gamma}_{n_0}:=
\gamma_{n_0}$. We infer from $\bilip(\g_{n_0})\ge\frac1{L}$ 
(since $\g_{n_0}\in\mathcal{F}_{p,L,R}(\mathcal{K})$)
and
\eqref{eq:bilip-tilde-eta} in
Lemma \ref{lem:bilip-control-harmonic} for $\eta:=\g_{n_0}$,
that there is a constant $\tilde{L}\in [L,2L]$ such that
\begin{equation}\label{eq:bilip-tildegamma}
\textstyle
|\tilde{\g}_{n_0}(y)-\tilde{\g}_{n_0}(z)|\ge
\frac1{\tilde{L}}d_{\tilde{\g}_{n_0}}(y,z)
\quad\Foa y,z\in\R/\Z.
\end{equation}
Thus, $\tilde{\g}_{n_0}$ is embedded. Since $\g_{n_0}$ is assumed to be
of class $C^1(\R/\Z,\R^3)$, and we substituted the subarcs 
$\g_{n_0}([x_-^i,x_+^i])$ by straight segments with small
angels 
$\ANG\big(\g_{n_0}'(x_\pm^i),\frac{\g_{n_0}(x_+^i)-\g_{n_0}(x_-^i)}{
x_+^i-x_-^i}\big)$ for $i=1,\ldots,N$, we can modify the
arguments given by Crowell and Fox
in \cite[Appendix~I]{crowell-fox_1977} to show that 
the modified curve $\tilde{\g}_{n_0}$ is still tame\footnote{although 
$\tilde{\g}_{n_0}$ is in general not in $W^{\frac32,2}(\R/\Z,\R^3)$ 
anymore because of the harmonic substitution; 
cf.\@ Footnote~\ref{foot:harmonic-substitution}.}.
Let us briefly sketch their arguments and indicate the necessary
modifications, the details are then left to the reader.

For a given arclength parametrized $C^1$-knot $\eta$ and any angle
$\alpha\in (0,\alpha_0(\eta)]$, where the threshold angle $\alpha_0(\eta)$
depends on the modulus of continuity of the unit tangent $\eta'$,
Crowell and Fox construct a necklace of double cones (see
\cite[Figure 62]{crowell-fox_1977}) with apex angle $\alpha$,
such that the finitely many apices of these double cones form a closed
polygon $P$ inscribed in $\eta$, and such that $\eta$ intersects every
circular cross-section of each double cone exactly once
and is contained in that necklace. The inscribed
polygon $P$ does the same, which leads them to define a homeomorphism
of $\R^3$ which equals the identity outside the double cones, by
mapping within each cross-section the unique curve point of $\eta$ linearly
to the respective unique point on $P$ in that same cross-section;
see \cite[Figure 64]{crowell-fox_1977}. This way they show that $\eta$
and $P$ are equivalent to conclude that any such $C^1$-curve $\eta$
is tame. In our situation we proceed in the same way\footnote{Here it is important to let 
$m=\bilip\eta=\tilde L$ (not just $L$) 
in~\cite[Eq.~(6), p.~149]{crowell-fox_1977}.} 
with $\tilde{\eta}:=
\tilde{\g}_{n_0}|_{\R/\Z\setminus\bigcup_{i=1}^N [x_-^i,x_+^i]}$
obtaining $N$ open polygonal lines with endpoints $\tilde{\g}_{n_0}(x_\pm^i)$,
$i=1,\ldots,N$, forming the axes of $N$ disjoint necklaces each consisting
of mutually disjoint double cones (by choosing a suitably small 
apex angle $\alpha$ and sufficiently small distances between
neighboring apices). Taking the union of these $N$ polygonal lines with the
$N$ straight segments $\tilde{\g}_{n_0}([x_-^i,x_+^i]), $ $i=1,\ldots,N$,
defines a closed polygon $P$ inscribed in $\tilde{\g}_{n_0}$ (coinciding
with $\tilde{\g}_{n_0}$ in its straight segments). 
We will now show that for any
 $\alpha\in (0,\frac{\pi}4)$ one has
\begin{equation}\label{eq:small-beta-angle}
\beta:=\max_{i=1,\ldots,N}\textstyle
\max\big\{\ANG\big(\g_{n_0}'(x_-^i),\frac{\g_{n_0}(x_+^i)-\g_{n_0}(x_-^i)}{
x_+^i-x_-^i}\big),
\ANG\big(\frac{\g_{n_0}(x_+^i)-\g_{n_0}(x_-^i)}{
x_+^i-x_-^i},\g_{n_0}'(x_+^i)\big)
\big\} < \pi -2\alpha,
\end{equation}
to prevent intersections of any segment $\tilde{\g}_{n_0}([x_-^i,x_+^i])$
with their neighboring double cones. 
Indeed, \eqref{eq:small-beta-angle} is guaranteed by inequality
\eqref{eq:difference-quotients} of Lemma \ref{lem:difference-quotients},
which implies for $\nu_i:=\nu_{B_{\bar{r}}(x_i)}\in\S^{n-1}$ defined
as in \eqref{eq:bmo3},
\begin{align*}
\cos\ANG\big(\g_{n_0}'(x_\pm^i),\nu_i\big)& =
\textstyle
\langle\g_{n_0}'(x_\pm^i),\nu_i
\rangle_{\R^3}=\lim_{s\to 0}\big\langle 
\frac{\g_{n_0}(x_\pm^i+s)-\g_{n_0}(x_\pm^i)}{s},\nu_i\big\rangle_{\R^3}
\overset{\eqref{eq:difference-quotients}}{\ge}
1-2\theta^\frac14,
\end{align*}
and also
\begin{align*}
\textstyle
\cos\ANG\big(\frac{\g_{n_0}(x_+^i)-\g_{n_0}(x_-^i)}{
|\g_{n_0}(x_+^i)-\g_{n_0}(x_-^i)|},\nu_i\big)& =
\big\langle
\frac{\g_{n_0}(x_+^i)-\g_{n_0}(x_-^i)}{x_+^i-x_-^i},\nu_i
\big\rangle_{\R^3}\cdot\frac{x_+^i-x_-^i}{|\g_{n_0}(x_+^i)-\g_{n_0}(x_-^i)|}
\overset{\eqref{eq:difference-quotients}}{\ge}
1-2\theta^\frac14.
\end{align*}
With $\theta<\theta_3\le 64^{-8}$ we find $1-2\theta^\frac14\ge 1-\frac1{128}
>\frac{\sqrt{3}}2,$ so that $\beta
<\frac{\pi}3,$
which implies \eqref{eq:small-beta-angle}.
In order to show that no segment $\tilde{\g}_{n_0}([x_-^i,x_+^i])$
interferes with other parts of the curve, we argue as in~\cite[(I.6), p.~150]{crowell-fox_1977}.

Notice that if $V_\varepsilon\not=\emptyset$ with  some point
$\xi_1\in V_{\varepsilon},$ then by symmetry of $\g_{n_0}$ also the 
points $\xi_j:=\xi_1+\frac{j-1}p\in\R/\Z$ for $j=2,\ldots,p$, are contained
in $V_\varepsilon$, 
so that $\sharp V_\varepsilon\ge 2$.

We are now going to employ the harmonic substitution.
If we replace a part of a curve by a straight line, we may face both global and local effects.
A global effect means that a strand of $\g|_{\R/\Z\setminus B_{\bar r}(x_i)}$ might interfere
with the substitution.
This means that there are some $\xi\in[x_i^-,x_i^+]$, $z\in \R/\Z\setminus B_{\bar r}(x_i)$
such that $\g_{n_0}(z)$ lies on the closed
segment with endpoints $\g_{n_0}(\xi)$ and 
$\tilde\g_{n_0}(\xi)$. 
(We can disregard points in $[x_i-\bar r,x_i^-]$ and $[x_i^+,x_i+\bar r]$ due to small local distortion.)
But in this case we find
\[ \textstyle \frac{3\bar r}{8L}\le \frac1Ld_{\g_{n_0}}(z,\xi)\le |\g_{n_0}(z)-\g_{n_0}(\xi)|
\le |\tilde\g_{n_0}(\xi)-\g_{n_0}(\xi)| \le \|\tilde\g_{n_0}-\g_{n_0}\| \le 6\theta^{1/8}\bar r \]
which implies $1\le\frac83L\cdot6\theta^{1/8}=\frac1{48}$, a contradiction.
It remains to consider local effects.
So,
if any of the subarcs $\g_{n_0}([x_-^i,x_+^i])$ (trivially extended to
infinity) were non-trivially knotted, then by symmetry of $\g_{n_0}$
there are at least two such non-trivially knotted subarcs, which implies
that $\g_{n_0}$ is a composite knot, contradicting our assumption.
Consequently, the subarcs $\g_{n_0}([x_-^i,x_+^i])$ are topologically
trivial for $i=1,\ldots,N$, and these subarcs were replaced by the
topologically trivial straight segments $\tilde{\g}_{n_0}(
[x_-^i,x_+^i])$, which implies that $\g_{n_0}$ is equivalent to
$\tilde{\g}_{n_0}$, i.e., $[\g_{n_0}]=[\tilde{\g}_{n_0}]$.

{\it Step 6. The final local distortion estimate for
the modified curve $\tilde{\g}_{n_0}$.}\,
For any $\xi\in\R/\Z$ one either has $\xi\in\overline{B_{\frac34 \bar{r}}(
V_\varepsilon)}$, in which case there exists $i\in\{1,\ldots,N\}$ 
with $x_i\in V_\varepsilon$, such that
\begin{equation}\label{eq:xi-contain1}
B_{\frac18\bar{r}}(\xi)\subset 
B_{\bar{r}}(x_i),
\end{equation}
or, by the covering \eqref{eq:covering-complement-Veps}, there exists
$k\in\{1,\ldots,K\}$ such that $\xi\in B_{\frac14 \bar{r}}(y_k)$, so that
\begin{equation}\label{eq:xi-contain2}
B_{\frac18\bar{r}}(\xi)\subset B_{\frac38 \bar{r}}(y_k).
\end{equation}
If \eqref{eq:xi-contain1} holds true then $B_{\frac18\bar{r}}(\xi)$
is either contained in the interval $[x_i-\bar{r},x_+^i]$ or in
$[x_-^i,x_i+\bar{r}]$ so that we may apply \eqref{eq:linfty-dist-theta}
of Lemma \ref{lem:harmonic-substitution} to the curve $\eta:=
\g_{n_0}$, to the parameters $x_1,\ldots,x_N\in V_\varepsilon$,
the radius $r:=\bar{r}$, and $\theta=\theta_4\in (0,\theta_2)$ (see
\eqref{eq:fix-theta}), to obtain
\begin{equation}\label{eq:distortion-tildegamma1}
\sup_{y,z\in B_{\frac18 \bar{r}}(\xi)\atop y\not= z}\textstyle
\frac{d_{\tilde{\g}_{n_0}}(y,z)}{|\tilde{\g}_{n_0}(y)-
\tilde{\g}_{n_0}(z)|}< 1+4\theta^\frac14<\frac{\pi}3<\frac{\pi}{\sqrt{8}}
<g_3=\frac{\sqrt{32}}{\sqrt{27}}\cdot
\frac{\pi}{\sqrt{8}}.
\end{equation}
If, on the other hand, \eqref{eq:xi-contain2} is true, then we know
from \eqref{eq:covering-intersect} and our construction of 
$\tilde{\g}_{n_0}$,  which implies in particular that $\tilde{\g}_{n_0}=
\g_{n_0}$ on $\R/\Z\setminus B_{\frac58 \bar{r}}(V_\varepsilon)$,
that
\begin{equation}\label{eq:coincidence-gamma-tildegamma}
\tilde{\g}_{n_0}=\g_{n_0}\quad\ON B_{\frac38 \bar{r}}(y_k).
\end{equation}
By means of \eqref{eq:summary-n0} we infer
\begin{equation}\label{eq:small-mu-n0-measure}
\lfloor\g_{n_0}'\rfloor^2_{\frac12,2,B_{\frac38 \bar{r}}(y_k)}=
\mu_{n_0}\big(B_{\frac38 \bar{r}}(y_k)\times B_{\frac38 \bar{r}}(y_k)\big)
\le 3\varepsilon,
\end{equation}
which we combine with \eqref{eq:seminorm-bilip2} in Corollary
\ref{cor:seminorm-bilip} for $\g:=\g_{n_0}$ to conclude
\begin{align}
\label{eq:distortion-tildegamma2}
\sup_{y,z\in B_{\frac18 \bar{r}}(\xi)\atop y\not= z}
\textstyle
\frac{d_{\tilde{\g}_{n_0}}(y,z)}{|\tilde{\g}_{n_0}(y)-
\tilde{\g}_{n_0}(z)|}&
\overset{\eqref{eq:xi-contain2},\eqref{eq:coincidence-gamma-tildegamma}}{\le}
\sup_{y,z\in B_{\frac38 \bar{r}}(y_k)\atop y\not= z}
\textstyle
\frac{d_{\g_{n_0}}(y,z)}{|\g_{n_0}(y)-\g_{n_0}(z)|}
\overset{\eqref{eq:seminorm-bilip2},\eqref{eq:small-mu-n0-measure}}{\le}
\frac1{\sqrt{1-\frac32 \varepsilon}}
\overset{\eqref{eq:epsilon-choice}}{=}\frac{\pi}3<\frac{\pi}{\sqrt{8}}<g_3,
\end{align}
by our choice of $\varepsilon$ in \eqref{eq:epsilon-choice}.
In view of \eqref{eq:distortion-tildegamma1} and
\eqref{eq:distortion-tildegamma2} we have shown that, by continuity of
$\tilde{\g}_{n_0}$,
\begin{equation}\label{eq:distortion-tildegamma}
\sup_{y,z\in\R/\Z\atop 0<d_{\tilde{\g}_{n_0}}(y,z)\le\frac14\bar{r}}
\textstyle
\frac{d_{\tilde{\g}_{n_0}}(y,z)}{|\tilde{\g}_{n_0}(y)-
\tilde{\g}_{n_0}(z)|}\le\frac{\pi}3< g_3.
\end{equation}
According to \eqref{eq:bilip-tildegamma}  we have for two parameters
$y,z\in\R/\Z$ with intrinsic distance $d_{\tilde{\gamma}_{n_0}}(y,z)
>\frac14\bar{r}$ the estimate
$$
\textstyle
\frac14\bar{r}<d_{\tilde{\gamma}_{n_0}}(y,z)\le\tilde{L}
|\tilde{\g}_{n_0}(y)-
\tilde{\g}_{n_0}(z)|\le 2L|\tilde{\g}_{n_0}(y)-
\tilde{\g}_{n_0}(z)|.
$$
Therefore, \eqref{eq:distortion-tildegamma} implies for the local
distortion of $\tilde{\g}_{n_0}$
\begin{equation}
\textstyle\delta(\tilde{\g}_{n_0},r_{\tilde{\g}_{n_0}})\le\frac{\pi}3<\frac{\pi}{\sqrt{8}}<g_3
\end{equation}
at scale $r_{\tilde{\g}_{n_0}}:=\frac1{16L}\bar{r}.$ In order to apply
Corollary \ref{cor:equiv-knots} to $\g_1:=\g$ and $\g_2:=\tilde{\g}_{n_0}$ 
it suffices to verify by means of \eqref{eq:summary-n0} and 
\eqref{eq:linfty-dist-theta} in Lemma \ref{lem:harmonic-substitution}
for $\eta:=\g_{n_0}$
\begin{eqnarray*}
\|\g-\tilde{\g}_{n_0}\|_{L^\infty}&\le &\|\g-\g_{n_0}\|_{L^\infty}+
\|\g_{n_0}-\tilde{\g}_{n_0}\|_{L^\infty}\\
&\overset{\eqref{eq:summary-n0},\eqref{eq:linfty-dist-theta}}{<}&
\textstyle
\frac1{128L}\bar{r}+6\theta^\frac18\bar{r}
\overset{\eqref{eq:fix-theta}}{=}\frac1{64L}\bar{r}\le\frac14\min\{
r_\g,r_{\tilde{\g}_{n_0}}\}.
\end{eqnarray*}
Therefore, we deduce $\mathcal{K}=
[\g_{n_0}]=[\g]$ by Corollary \ref{cor:equiv-knots}. \hfill $\Box$

\section{Symmetric critical knots}\label{sec:symm-critical}
In this section we are going to use our weak compactness result, Theorem
\ref{thm:fractional-compactness}, to prove the existence of rotationally
symmetric critical knots for the M\"obius energy.

{\it Proof of Theorem~\ref{thm:symm-crit-prime}.}\,
Since $\Sigma_p(\mathcal{K})$ contains a knot with finite M\"obius energy,
the knot equivalence class $\mathcal{K}$ is tame according to
\cite[Theorem~4.1]{freedman-etal_1994}, 
the infimum of the energy is finite, and there is a minimizing sequence
$(\g_m)_m\subset\Sigma_p(\mathcal{K})$ with 
$$
\emob(\g_m)\to\inf_{\Sigma_p(\mathcal{K})}\emob(\cdot)<\infty\quad
\As m\to\infty.
$$
Due to parameter invariance of $\emob$ and because
rescaling to unit length and subsequent
reparametrization to arclength preserves the knot equivalence class,
and also the $p$-rotational symmetry 
\eqref{eq:rot-symm} according to 
\cite[Lemma 4.6]{gilsbach-vdm_2018}\footnote{see
 version 2 in \cite[Lemma 4.6]{gilsbach-vdm_2017} for the corrected proof.},
we may assume without loss
of generality that $|\g_m'|=1$  a.e.\@ on $\R/\Z$ for all $m$. In order
to apply Theorem~\ref{thm:fractional-compactness} we use convolutions
$\g_{m,\epsilon_m}\in C^\infty(\R/\Z,\R^3)$, $\epsilon_m\to 0 $ as 
$m\to\infty$, 
their
rescalings $\hat{\g}_{m,\epsilon_m}:=(\mathcal{L}(\g_{m,\epsilon_m}))^{-1}
\g_{m,\epsilon_m}$, and reparametrize $\hat{\g}_{m,\epsilon_m}$ to 
arclength to obtain $\G_m\in C^\infty(\R/\Z,\R^3)$ with $|\G_m'|=1$
a.e.\@ on $\R/\Z$ for all $m\in\N$, such that
$$
\emob(\G_m)\le\emob(\g_m)+\frac1{m}\quad\Foa m\in\N,
$$
which is possible due to
\cite[Theorems 1.1, 1.3 \& 1.4]{blatt_2019b}. Especially 
\cite[Theorem~1.3]{blatt_2019b} in combination with Corollary 
\ref{cor:fractional-stability}
of the present paper allows us to assume that $[\G_m]=[\g_m]=
\mathcal{K}$ for all $m\in\N$.
As before, this mollification,
 rescaling, and
 subsequent reparametrization to arclength preserves the 
 $p$-rotational
 symmetry defined in \eqref{eq:rot-symm},
 so that $\G_m\in\Sigma_p(\mathcal{K})$ for  all $m\in\N$.

 This new $\emob$-minimizing sequence $(\G_m)_m$ has
 equibounded M\"obius energy, which yields a uniform bound on the 
 seminorms $\lfloor \G_m'\rfloor_{\frac12,2}$ according to
  \cite[Theorem~1.1]{blatt_2012a}
 or \cite[Theorem~3.2(i)]{gilsbach-vdm_2018}, as well as a uniform bilipschitz constant
 $L$ depending only on the uniform energy bound; see 
 \cite[Theorem~2.3]{ohara_1992a}.
 So, for each $m\in\N$, the curve $\G_m$ is contained in the class
 $\mathcal{F}_{R,L,p}$ as defined in 
 \eqref{eq:compact-subset}
 for $R:=\sup_{m\in\N}\lfloor \G_m'\rfloor^2_{\frac12,2}$,
 where we assumed (by translational invariance of $\emob$) that 
 $\G_m(0)=0\in\R^3$ for all $m\in\N$.

Consequently, Theorem~\ref{thm:fractional-compactness} implies the existence
of a limiting arclength parametrized knot $\G\in\mathcal{F}_{R,L,p}$ such that
a subsequence (still denoted by $\G_m$) converges weakly in 
$W^{\frac32,2}$ and uniformly to $\G$ as $m\to\infty$.
The 
lower semicontinuity of $\emob$ (see \cite[Lemma 4.2]{freedman-etal_1994})
implies
\begin{equation}\label{eq:inf-emob}
\inf_{\Sigma_p(\mathcal{K})}\emob(\cdot)
\le\emob(\G)\le\liminf_{m\to\infty}\emob(\G_m)=\inf_{\Sigma_p(\mathcal{K})}
\emob(\cdot);
\end{equation}
hence $\G$ is the desired symmetric minimizing prime knot.

It remains to be verified that $\G$ is critical for the M\"obius energy on its full domain, because then $\G\in C^\omega(\R/\Z,\R^3)$ due to
\cite[Corollary 1.3]{blatt-vorderobermeier_2019}.
For that notice first that we can proceed exactly as
in the proof of Corollary \ref{cor:fractional-stability}
to find a scale $r_\G>0$ for
the arclength parametrized $W^{\frac32,2}$-knot $\G$, such that
$\delta(\G,r_\G)< g_3$.
We symmetrize any variational vector $h\in C^\infty(\R/\Z,\R^3)$ according to
\begin{equation}\label{eq:symmetrize}
\textstyle
h_\textnormal{sym}(u):=\sum_{k=1}^p\Rot_{-\frac{2\pi k}p}\circ h(u+\frac{k}p)
\quad\Fo u\in\R/\Z,
\end{equation}
to obtain by virtue of Corollary \ref{cor:lipschitz-stability}
a number $\epsilon_{\G,h}>0$ such that the perturbed curves
$\Upsilon_\epsilon:=\G+\epsilon h_\textnormal{sym}$ are immersed knots
equivalent to $\G$ for all $|\epsilon|\le\epsilon_{\G,h}.$
Regarding $p$-rotational symmetry we compute (using \eqref{eq:rot-symm} for
$\G$)
\begin{align*}
\textstyle
&\Rot_{\frac{2\pi}p}\circ\Upsilon_\epsilon(u-\frac1{p})
\textstyle
=\Rot_{\frac{2\pi}p}\circ\G(u-\frac1p)+\epsilon\Rot_{\frac{2\pi}p}\circ
h_\textnormal{sym}(u-\frac1p)\\
&\textstyle\overset{\eqref{eq:rot-symm}}{=}
\G(u)+\epsilon\sum_{k=1}^p\Rot_{-\frac{2\pi}p(k-1)}\circ
h(u+\frac{k-1}p)
= \G(u)+\epsilon\sum_{k=1}^{p-1}\Rot_{-\frac{2\pi k}p}
\circ h(u+\frac{k}p)+\epsilon h(u)\\
&=\G(u)+\epsilon\sum_{k=1}^p
\Rot_{-\frac{2\pi k}p}\circ h(u+\frac{k}p)
 = \G(u)+\epsilon h_\textnormal{sym}(u)=\Upsilon_\epsilon (u)\quad\Foa
u\in\R/\Z,\,|\epsilon|\le\epsilon_{\G,h}.
\end{align*}
Hence $\Upsilon_\epsilon\in\Sigma_p(\mathcal{K})$ and \eqref{eq:inf-emob}
implies
$$
\emob(\G)\le
\emob(\Upsilon_\epsilon)=
\emob(\G+\epsilon h_\textnormal{sym})\quad\Foa |\epsilon|\le\epsilon_{\G,h}.
$$
Therefore, taking the first variation $\delta\emob(\G,\cdot)$, which is linear in 
its second entry,
\begin{equation}\label{eq:variation}
\textstyle
0=\delta\emob(\G,h_\textnormal{sym})=\sum_{k=1}^p\delta\emob\big(\G,
\Rot_{-\frac{2\pi k}p}\circ h (\cdot +\frac{k}p)\big).
\end{equation}
Furthermore, from the symmetry \eqref{eq:rot-symm}
of $\G$ and the invariance of $\emob$ under
rotations and reparametrizations one deduces 
for every $k\in\{1,\ldots,p\}$
\begin{align*}
\textstyle
\emob\big(&\textstyle\G(\cdot)+\epsilon\Rot_{-\frac{2\pi k}p}\circ h(\cdot+\frac{k}p)\big)
\overset{\eqref{eq:rot-symm}}{=}
\textstyle
\emob\big(\Rot_{-\frac{2\pi k}p}\circ\G(\cdot +\frac{k}p)
+\epsilon\Rot_{-\frac{2\pi k}p}\circ h(\cdot+\frac{k}p)\big)\\
& \textstyle
=\emob\big(\Rot_{-\frac{2\pi k}p}\circ [\G(\cdot +\frac{k}p)+\epsilon
h(\cdot+\frac{k}p)]\big)=
\emob\big(\G(\cdot +\frac{k}p)+\epsilon
h(\cdot+\frac{k}p)\big)=\emob(\G+\epsilon h),
\end{align*}
which implies 
$$
\textstyle
\delta\emob\big(\G,\Rot_{-\frac{2\pi k}p}\circ h(\cdot +
\frac{k}p)\big)
\!=\!\frac{d}{d\epsilon} |_{\epsilon=0}
\emob\big(\G+\epsilon\Rot_{-\frac{2\pi k}p}\circ h(\cdot +\frac{k}p)
\big)\!=\!
\frac{d}{d\epsilon} |_{\epsilon=0}
\emob(\G+\epsilon h)\!=\!\delta\emob(\G,h)
$$
for all $k=1,\ldots,p$.
Therefore, by \eqref{eq:variation},
$$
\textstyle
0=\sum_{k=1}^p\delta\emob\big(\G,\Rot_{-\frac{2\pi k}p}\circ h(
\cdot +\frac{k}p)\big)=p\delta\emob(\G,h),
$$
which leads to the desired criticality, since by continuity of the first variation
we may approximate an arbitrary variational vector field $h\in W^{\frac32,2}(
\R/\Z,\R^3)$ by smooth vector fields, for which the first variation
has just been shown to vanish.

In the specific case of torus knot classes $\mathcal{K}=\mathcal{T}(a,b)$
for co-prime integers $a,b\in\Z\setminus\{0,\pm 1\}$ we can proceed
exactly as in the proof of \cite[Theorem~1.2]{gilsbach-vdm_2018}
to show that there
are at least two different $\emob$-critical torus knots
$\G_1,\G_2$ in $\mathcal{T}(a,b)$ for each 
co-prime integers $a,b\in\Z\setminus\{0,\pm 1\}$. The only prerequisite
for that proof is the existence of symmetric energy
minimizers for different
rotational symmetries, which can be  obtained as above for $\emob$ on 
the symmetric subsets
$\Sigma_{p_i}(\mathcal{T}(a,b))$, $i=1,2$, for $p_1=a$ and $p_2=b$.
\hfill $\Box$






\subsection*{Acknowledgements}

S.B.\@ was partially funded through FWF grant no. P 29487-N32 ``Gradient flows of curvature energies''.
A.G.\@ gratefully acknowledges funding through an incoming grant of the University
of Salzburg, and through the stipend  ``International Research Fellow of JSPS (Postdoctoral Fellowships for Research in Japan)''.
Ph.R.\@ has been partially supported by the
German Research Foundation (DFG) through project number 289032105.
H.v.d.M.\@'s work is partially funded by the Excellence Initiative of the
German federal and state governments.

While working on an early draft of this paper, A.G.\@ 
enjoyed the hospitality of the Department of Mathematics at the University of Georgia, Athens, GA.

We are indebted to the students who attended H.v.d.M.'s 
lectures
on this topic  in the winter term 2024-2025 at 
RWTH Aachen University for their valuable comments and
corrections.

\bibliographystyle{abbrvhref}
\bibliography{Reiter-refs}

%
%
%

\end{document}